\documentclass[a4paper,11pt,reqno]{article}
\usepackage{a4wide}
\usepackage[english]{babel}
\usepackage{amssymb}
\usepackage{amsmath}
\usepackage{amsfonts}
\usepackage{amsthm}
\usepackage{hyperref}
\hypersetup{colorlinks,citecolor=red,filecolor=purple,linkcolor=blue,urlcolor=black}
\usepackage[T1]{fontenc}
\usepackage{enumerate}
\usepackage{relsize}

\newcommand{\equinf}[1]{\ \underset{#1\to+\infty}{\mathlarger{\mathlarger{{\sim}}}}\ }
\newcommand{\Nat}{\mathbb N}
\newcommand{\R}{\mathbb R}
\newcommand{\Esp}{\mathbb E}
\newcommand{\p}{\mathbb P}
\newcommand{\Einf}[1]{\Esp_\infty\left(#1\right)}
\newcommand{\Vinf}[1]{\Var_\infty\left(#1\right)}
\newcommand{\pinf}[1]{\p_\infty\!\left(#1\right)}
\newcommand{\linf}[1]{\underset{#1\to+\infty}\longrightarrow}
\newcommand{\1}[1]{\mathbf{1}\!_{\left\{#1\right\}}}
\newcommand{\EE}[1]{\Esp\left(#1\right)}
\newcommand{\Ek}[1]{\Esp_k\left(#1\right)}
\newcommand{\as}{\quad\mathrm{ a.s.}}
\newcommand{\eps}{\varepsilon}
\newcommand{\dif}{\mathrm{d}}
\newcommand{\Var}{\mathrm{Var}}
\newcommand{\me}{\medskip \noindent}
\newcommand{\bi}{\bigskip \noindent}

\def\be{\begin{eqnarray}}
\def\ee{\end{eqnarray}}
\def\ben{\begin{eqnarray*}}
\def\een{\end{eqnarray*}}

\newtheorem{prop}{Proposition}[section]
\newtheorem{defi}[prop]{Definition}

\newtheorem{lem}[prop]{Lemma}
\newtheorem{thm}[prop]{Theorem}

\newcounter{example}
\newenvironment {example}
		{\refstepcounter{example}
		 {\noindent\emph{Example \arabic{example}.} }}
		{\medskip}

\title{How do  birth and death processes come down from infinity?}

\author{Vincent Bansaye\thanks{CMAP, Ecole Polytechnique, CNRS, route de
    Saclay, 91128 Palaiseau Cedex-France; E-mail: \href{mailto:vincent.bansaye@polytechnique.edu}{\texttt{vincent.bansaye@polytechnique.edu}}}, Sylvie M\'el\'eard\thanks{CMAP, Ecole Polytechnique, CNRS, route de
    Saclay, 91128 Palaiseau Cedex-France; E-mail: \href{mailto:sylvie.meleard@polytechnique.edu}{\texttt{sylvie.meleard@polytechnique.edu}}},
     Mathieu Richard\thanks{CMAP, Ecole Polytechnique, CNRS, route de
    Saclay, 91128 Palaiseau Cedex-France; E-mail: \href{mailto:mathieu.richard@cmap.polytechnique.fr}{\texttt{mathieu.richard@cmap.polytechnique.fr}}}}

\begin{document}
\maketitle

\begin{abstract}
We finely describe the "coming down from infinity" for  birth and death processes which eventually become extinct.  Our biological motivation is   to study the decrease of regulated populations  which are initially large.
 Under general assumptions on the birth and death rates, we describe the  behavior of the hitting time of large integers. We let two regimes appear and   derive an  expression of the  speed of coming down from infinity. In the case of  death rates  with regular variations, we also get a central limit theorem and the asymptotic  probability of extinction in small times.  
Finally, we apply our results to  birth and death processes in varying environment in whose the environment influences the competition.
\end{abstract}



\section{Introduction}

Our goal in this paper is to finely describe the "coming down from infinity" for a birth and death process which eventually becomes extinct. Our motivations come from the study of population dynamics with  initially large populations. In particular we wish to describe the effect of the competition in large populations and  specify persistence criteria in a possibly varying environment. For this purpose, we decompose  the trajectory of the process with respect to the hitting times of large integers and we use some asymptotic results about sums of independent random variables and Tauberian theorems. 

\me
The population size is modeled by  a birth and death process   $(X(t),t\geq0)$  whose birth rate (resp. death rate) at state $n\in \mathbb{N}$ is $\,\lambda_n$ (resp. $\mu_n$).
In the whole paper, we assume that $\lambda_n$ are nonnegative  and $\mu_n$ are positive  for $n\geq1$ and that   $\mu_0=\lambda_0=0$. The latter implies that $0$ is an absorbing state.
These processes have been extensively studied from the pioneering works on birth and death processes \cite{KarlinMcGregor57} and on the quasi-stationary distribution \cite{Doorn1991}.

\me
It is well known \cite{KarlinMcGregor57, Karlin1975}  that
\begin{equation}\label{condition_extinction}
\sum_{i\geq1}\frac{1}{\lambda_i\pi_i}=\infty
\end{equation}
is a necessary and sufficient 
 condition for  almost sure absorption of the process at $0$, where  for $n\geq1$,
\begin{equation*}
\label{defi_pi}
\pi_n:=\frac{\lambda_1\cdots\lambda_{n-1}}{\mu_1\cdots\mu_n}.
\end{equation*}

\me 
Under \eqref{condition_extinction}, we let the initial population size go to infinity and focus on the case where the limiting process comes back to finite values in finite time. This behavior is captured
by the notion of "coming down from infinity".   Characterizations of the coming down from infinity have been given in \cite{Anderson1991, Cattiaux2009}. They rely on  the convergence of the  series  
\be
\label{cv_{S}}
S:=\sum_{n=1}^{+\infty}\frac{1}{\lambda_n\pi_n}\sum_{i\geq n+1}\pi_i \ < + \infty
\ee
or on the finiteness of the first moment of time of absorption, uniformly in the initial condition.
  As proved in Van Doorn \cite{Doorn1991}, this is also equivalent to the existence and uniqueness of the quasi-stationary distribution  at $0$.  In particular, the uniqueness of the quasi-stationary distribution is deeply related to the way the process comes back into compact sets,  see \cite{meleard_villemonais} or \cite{Cattiaux2009}.

\me
In Section \ref{section_CDI}, we improve this result by an additional exponential moment condition linked to the Lyapounov function $J(n)= \sum_{k=1}^{n-1}\frac{1}{\lambda_k\pi_k}\sum_{i\geq k+1}\pi_i$.

\me We 
 go further in the description of the coming down from infinity, under a slightly more restrictive  
  assumption than \eqref{condition_extinction} : 
\begin{equation}
 \label{condition_limite}
\frac{\lambda_n}{\mu_{n}}\linf{n}l<1.
\end{equation}
 \noindent That allows us to rigorously define  the law $\mathbb{P}_{\infty}$ of the process starting from infinity by a tightness argument. Assumption \eqref{condition_limite} is satisfied by the parameters of the classical models motivated by ecology, including competition between individuals or Allee effect. 
 We also need the following technical assumption  on the death rate to obtain the asymptotic behavior of the integers hitting times : 
\begin{equation}\label{condition_croissance}
\sup_{k,n\geq1}\frac{\mu_n}{\mu_{n+k}}<\infty.
\end{equation}
Two interesting classes of sequences fulfill condition \eqref{condition_croissance}: the death rates which are non-decreasing for large enough integers and the regularly varying death rates (see Section \ref{fonctions_variation_reguliere} in Appendix  for  definitions). \\
When conditions \eqref{condition_limite} and \eqref{condition_croissance} are both satisfied, it is easy to check (see Lemma \ref{equivalence_CDI} in the Appendix) that 
 the series $S$ is finite if and only if 
\begin{equation}\label{condition_CDI_mu}
\sum_{i\geq1} \frac{1}{\mu_i}<+\infty.
\end{equation}

Under the assumptions  \eqref{condition_limite}, \eqref{condition_croissance}  and \eqref{condition_CDI_mu}, we study in Section \ref{section_comportement_T_n} the asymptotic behavior of the decreasing sequence $\,(T_{n})_{n}\,$ of hitting times, defined by
 $$
 T_{n}= \inf\{t\geq 0, X(t)=n\}.$$
Then, \eqref{condition_extinction} and \eqref{cv_{S}} are satisfied, $X$ comes down from infinity a.s. at time $0$. We put in light two different regimes which depend on whether the mean time to go from  $n+1$ to $n$  is negligible or not compared to the mean time to reach $n$. We are then able to give the asymptotic behavior of $X$ for small times.  We show that the  speed of coming down from infinity is obtained from a deterministic decreasing function $t\mapsto v(t)$  tending to infinity at $0$ and  defined as the generalized inverse of the mapping $n\mapsto \mathbb{E}_{\infty}(T_{n}) = S - J(n-1)$.
More precisely, our main result (Section \ref{SCDI}) ensures that 
  \begin{equation*}
  \lim_{t\to0}\frac{X(t)}{v(t)}=1,
 \end{equation*}
where the convergence is either  in probability  or almost surely,  depending on the respective asymptotic behaviors of the birth and death rates. For that, 
   we need to control the trajectory of the process between two successive times $T_n$. We also require that  the rates $\mu_n$ are regularly varying to
   get the a.s. convergence. In addition, we derive in that case a central limit theorem (Theorem \ref{TLC_X}) and  the probability to be absorbed exceptionally fast (Theorem  \ref{descente_rapide}). The proofs rely on  a central limit theorem  for the sum of independent random variables, some Tauberian results  and  coupling arguments. Roughly speaking, we prove that
  $$  \mathbb P_{\infty}(T_0\leq t) \quad \hbox{ behaves as } \ \exp\left(-t^{\frac{1}{1-\rho}}\right),$$
  as $t$ tends to $0$,
where $\rho$ is the index of the variations of $(\mu_n)_n$. 

These results apply in particular to the logistic birth and death process and to  the Kingman coalescent. 
In both cases, we  improve the known results on the coming down from infinity. \\
Lambert \cite{Lambert2005} characterizes the distribution of the  absorption time for the logistic branching process  starting from infinity. Our work extends  to very general death rates, as polynomial increase, which are motivated by ecological data for competition of species, see e.g. Sibly and al \cite{Sibly22072005}.\\
 The proof of the speed of coming down from infinity for Kingman coalescent  has already been obtained in Aldous \cite{Aldous99}.
We complete this result by    estimating the probability that the most recent common ancestor is achieved very fast. Our proof also suggests the way this rare event occurs by considering the associated successive coalescent times.  

Section \ref{application} is devoted to the main application of our results, which is an extinction criterion for time inhomogeneous birth and death processes. These processes have been studied in the framework of randomly varying environment, as described in Cogburn and Torrez  \cite{Cogburn1981}, \cite{Torrez1978}.
Our results allow to obtain extinction results in cases where the environment can be both unfavorable during some periods and favorable during the rest of the time. We quantify the minimal duration of the unfavorable environmental periods leading to eventual extinction. For example, this problem is   relevant  in epidemiology, when the environment influences the parameters of the infection (see Bacaer-Dads \cite{BacaerDads}
and 
van den Broek-Heesterbeek  \cite{BroekHees}). The proof relies on the evaluation of the probability of extinction for time homogeneous birth and death processes  starting from $\infty$ given  above.
\me

The paper is organized as follows. In the next section, we work under the extinction assumption \eqref{condition_extinction} and gather general characterizations of the coming down from infinity, such as \eqref{cv_{S}}. 
Focusing on the subclass of birth and death processes satisfying \eqref{condition_limite}, \eqref{condition_croissance}  and \eqref{condition_CDI_mu} in Section \ref{section_comportement_T_n}, we  describe the hitting time  of large integers
when the process starts from infinity. Under some additional assumption, we can then derive in Section \ref{section_comportement_X} the asymptotic behavior of $X(t)$ when $t$ is small. Moreover, when the death rate has regular variations, we can also
quantify the probability of extinction before a small time (Section \ref{section_fast_extinction}) and the minimal time of competition leading to extinction in varying environment (Section \ref{application}). Finally in Appendix, we illustrate our results by several examples and give some useful details on regularly varying functions.

\section{Coming down from infinity}\label{section_CDI}

\me We first focus on the time spent by the process $(X(t), t\geq 0)$  to go from level $n+1$ to level $n$. We introduce the notation
$$\tau_n:=\inf\{t>T_{n+1};X(t)=n\}-T_{n+1},
\quad G_n(a):=\EE{\exp(-a\tau_n)}, \  (a>0).$$
By the strong Markov property, $\tau_n$ has the law of $T_{n}$ under $\p_{n+1}$ and the random variables $(\tau_n)_{n\geq0}$ are independent.

\begin{prop}
\label{expect-taun}
For each $n\in \mathbb{N}$, we have
\be
\label{moment_{tn}}
\mathbb{E}(\tau_{n}) =  \Esp_{n+1}(T_{n}) = \frac{1}{\lambda_{n}\pi_{n}}\sum_{i\geq n+1}\pi_i
\ee
and for every $a>0$,
 \begin{equation}\label{recurrence_Gn}
G_{n}(a)=\mathbb{E}_{n+1}(\exp(-aT_n))=1+\frac{\mu_n+a}{\lambda_n}-\frac{\mu_n}{\lambda_n}\frac{1}{G_{n-1}(a)}.
 \end{equation}
\end{prop}

\begin{proof}  The proof of the first part uses the Lyapounov function
defined for all $m\in \mathbb{N}^*$ by
\begin{equation*}
\label{J}
J(m)=\left\{\begin{array}{cc}
\displaystyle\sum_{n=1}^{m-1}\frac{1}{\lambda_n\pi_n}\sum_{i\geq n+1}\pi_i& \ \textrm{if } m\geq 2,\\
0&\ \textrm{ if }m<2.
\end{array}\right.
\end{equation*}
We denote by $L$ the infinitesimal generator of $X$: for any bounded function $f$ on $\Nat$ and any $n\in\Nat^*$,
\begin{equation}
\label{defG}
L(f)(n)=\left(f(n+1)-f(n)\right)\lambda_n+\left(f(n-1)-f(n)\right)\mu_n.
\end{equation}
One easily checks that for any $m\in \mathbb{N}^*$,
$$L J(m) = -1.$$
Thus,  the process $$J(X(t))-\int_0^t LJ(X(u))\dif u = J(X(t)) + t$$
  is a martingale with respect to the natural filtration of $X$. Therefore, we have for all $k\geq0$ and $t\geq0$,
\begin{equation}\label{eq:1}
\mathbb{E}_{n+1}(J(X(t\wedge T_{n})))+\mathbb{E}_{n+1}(t\wedge T_{n}) =J(n+1).
\end{equation}
Adding that $J$ is bounded by $S<+\infty$, we can use  the bounded and  monotone convergence theorems to let $t\rightarrow \infty$ in \eqref{eq:1} and get 
$$ \mathbb{E}_{n+1}(J(X( T_{n})))+\mathbb{E}_{n+1}( T_{n}) = J(n) + \mathbb{E}_{n+1}( T_{n}) =J(n+1).$$
\noindent Thus $\mathbb{E}_{n+1}( T_{n}) =J(n+1)- J(n)$ which concludes the proof of the first part of the proposition. \\

\noindent We consider now the Laplace transform of $\tau_n$ and 
follow  \cite[p. 264]{Anderson1991}.
 By the Markov property, we have
 $$\tau_{n-1}\overset{(\mathrm{d})}=\1{Y_n=-1}E_n+\1{Y_n=1}\left(E_n+\tau_{n}+\tau'_{n-1}\right)$$
 where $Y_n$, $E_n$, $\tau'_{n-1}$ and $\tau_{n}$ are independent random variables,
  $E_n$ is exponentially distributed  with parameter $\lambda_n+\mu_n$ and $\tau'_{n-1}$ is distributed as $\tau_{n-1}$ and $\p(Y_n=1)=1-\p(Y_n=-1)=\lambda_n/(\lambda_n+\mu_n)$.
Hence, we get
 $$G_{n-1}(a)=\frac{\lambda_n+\mu_n}{a+\lambda_n+\mu_n}\left(G_n(a)G_{n-1}(a)\frac{\lambda_n}{\lambda_n+\mu_n}+\frac{\mu_n}{\lambda_n+\mu_n}\right)$$
 and \eqref{recurrence_Gn} follows.
\end{proof}
$\newline$

We now give the usual definition of coming down from infinity, which means that the state $\infty$ is an entrance boundary for the process $X$
\cite[p. 305]{Revuz1999}.
\begin{defi}
  We say that the process $(X(t), t\geq 0)$ comes down from infinity if there exist a positive number $t$ and a non-negative integer $m$ such that
  $$\lim_{k\to+\infty}\p_k(T_m<t)>0.$$
\end{defi}

\me  We give now several necessary and sufficient conditions for $(X(t), t\geq 0)$ to come down from  infinity. The two first ones are   directly taken  from \cite{Cattiaux2009} . We add here an exponential moment criterion which is useful for the forthcoming proofs.  Let us  also mention that  it is equivalent to the existence and uniqueness of a quasistationnary distribution (cf. Van Doorn \cite{Doorn1991}). 
\begin{prop}\label{CDI}
Under condition \eqref{condition_extinction}, the following assertions  are equivalent:
\begin{enumerate}[{\normalfont (i)}]
 \item The process $(X(t), t\geq 0)$ comes down from infinity.
\item $S<+\infty$.
\item $\sup_{k\geq0}\Esp_k [T_0]<+\infty$.
\item For all $a>0$, there exists $k_a\in\Nat$ such that $\sup_{k\geq k_a}\Esp_k\left(\exp(aT_{k_a})\right)<+\infty$.
\end{enumerate}
\end{prop}
\me
This result is the discrete counterpart of Lemma 7.4 in \cite{Cattiaux2009} for Feller diffusion processes
  $\dif Z_t=\sqrt{\gamma Z_t}dB_t+ Z_{t}(r- f(Z_t))\dif t$  and suitable    function $f$ and $r>0$.
Recall that if  \eqref{condition_limite} and \eqref{condition_croissance} are satisfied, Assertion (ii), and then (i), (iii), (iv),  are equivalent to Condition \eqref{condition_CDI_mu}, which can be seen as the discrete counterpart of the criteria in \cite[p.1953]{Cattiaux2009} stating that the process $Z$ comes down from infinity if and only if $\int_1^\infty\frac{\dif x}{xf(x)}<+\infty$.

\begin{proof}[Proof of Proposition \ref{CDI}]
 Assertions (i), (ii), and (iii) are equivalent according to \cite[Prop 7.10]{Cattiaux2009}. Let us now prove that (iv) implies that $X$ comes down from infinity. Indeed, taking $a=1$ in (iv), we have $M:=\sup_{k\geq k_1}\Esp_k\left(\exp(T_{k_1})\right)<+\infty$. Then Markov inequality ensures that  for all  $k\geq k_1$ and $t\geq0$,
$\p_k(T_{k_1}<t)\geq 1-\exp(-t)M$. Choosing  $t$ small enough ensures that the process comes down from infinity.

\noindent Finally, we   prove that (ii) implies (iv)  by adapting the proof of \cite[Prop 7.6]{Cattiaux2009}  to the discrete setting.
We fix $a>0$ and using  $S<+\infty$, there exists $k_a$ such that
\begin{equation*}\label{defi_n_a}\sum_{n\geq k_a-1}\frac{1}{\lambda_n\pi_n}\sum_{i\geq n+1}\pi_i\leq\frac{1}{a}.\end{equation*}
We now define the Lyapounov function  $J_a$ as
$$J_a(m):=\left\{\begin{array}{cc}
\displaystyle\sum_{n=k_a-1}^{m-1}\frac{1}{\lambda_n\pi_n}\sum_{i\geq n+1}\pi_i& \textrm{if } m\geq k_a\,,\\
0&\textrm{ if }m<k_a\,.
\end{array}\right.$$
We note that $J_{a}$ is  non-decreasing,   bounded and recalling the definition of the generator 
$L$ from \eqref{defG}, $LJ_a(m)=-1$ for any $ m\geq k_a$.
 Then,
$$M_t:=e^{at}J_a(X(t))-\int_0^te^{au}\left(aJ_a(X(u))+LJ_a(X(u))\right)\dif u, \qquad (t\geq0)$$
is a martingale with respect to the natural filtration of $X$.  Using 
the stopping time  $T_{k_a}$ and the fact that  $J_a(X(u))\leq J_a(\infty)\leq1/a$ , we have  for all $k\geq k_a$ and $t\geq0$,
\begin{eqnarray*}
\Ek{e^{at\wedge T_{k_a}}J_a(X(t\wedge T_{k_a}))}&=& \Ek{\int_0^{t\wedge T_{k_a}}e^{au}\left(aJ_a(X(u))+LJ_a(X(u))\right)\dif u}+J_a(k)\\
&= &\Ek{\int_0^{t\wedge T_{k_a}}e^{au}\left(aJ_a(X(u))-1\right)\dif u}+J_a(k)\\
&\leq & J_a(k)
\end{eqnarray*}
since  $ u\leq t\wedge T_{k_a}$ ensures that $X(u)\geq k_a$ and $LJ_a(X(u))=-1$. 
Therefore, for any  $k\geq k_a$, $\p_k$-a.s. $J_a(X(t\wedge T_{k_a}))\geq J_a(k_a)$ and
$$\Ek{e^{at\wedge T_{k_a}}}\leq \frac{J_a(k)}{J_a(k_a)}.$$
Then (iv) follows from the monotone convergence theorem and Assumption (ii). 
\end{proof}
\bi Under our additional assumption  \eqref{condition_limite},
 we can now define the process starting from infinity and check that it indeed comes down instantaneously from infinity a.s..
 We set $\overline{\Nat}:=\Nat\cup\{\infty\}$ and for any $T>0$, we denote by $\mathbb{D}_{\overline{\Nat}}([0,T])$ the Skorohod space of c\`adl\`ag functions on $[0,T]$ with values in $\overline{\Nat}$.
 
\begin{prop}\label{defi_P_infini} Let $T>0$. 
Under \eqref{condition_extinction} and \eqref{cv_{S}} and \eqref{condition_limite}, the law of  $X$ under $\p_k$ converges weakly as $k\to+\infty$ in $\mathbb{D}_{\overline{\Nat}}([0,T])$. \\
We denote by $\p_\infty$ the law of the limit and  we have
\begin{equation*}
\label{temps_fini_ps} \p_{\infty}\left(\inf\{t>0 : X(t)<+\infty\}=0\right)=1.
 \end{equation*}
 \end{prop}
\begin{proof} First, we show that under Assumption  \eqref{condition_limite} and  for any $N_{0}\in \mathbb{N}^*$ and $t>0$, 
\be
\label{moment}\sup_{s\in[0,t]} \Esp_{N_0}(X(s)) <+\infty.
\ee
Indeed,  there exists $ N\in \mathbb{N}^*$ such that for $n \geq N$, 
$\lambda_{n} - \mu_{n} \leq 0.$ Hence, for $s\leq t$,
\ben \Esp_{N_0}(X_{s}) &=& N_0+  \int_{0}^s \Esp_{N_0}\left(\lambda_{X(u)} - \mu_{X(u)} \right) \dif u\\
&\leq&   N_0+ \int_{0}^s\Esp_{N_0}\left((\lambda_{X(u)} - \mu_{X(u)}){\bf 1}_{X(u)\leq N} \right) \dif u\leq N_0+\sup_{n\leq N} |\lambda_{n} - \mu_{n}| t .
\een

 \noindent To prove the convergence of the sequence of laws  $\p_k$, we use  Theorem 1 in Donnelly \cite{Donnelly91}, which gives conditions under which a sequence of processes will converge (in law) to a Markov process with an entrance boundary. In our setting,
 the birth and death processes under $\p_{k}$ and $\p_{k'}$ only differ by their initial conditions $k$ and $k'$. Thus, we only need to check the equi-boundedness condition : for any $t>0$,
 \begin{equation}\label{lim_liminf}\lim_{A\to+\infty}\liminf_{k\to+\infty}\p_k(X(t)\leq A)=1.
 \end{equation}
For that purpose, we combine the first part of  Proposition \ref{expect-taun} and 
\eqref{cv_{S}} to get  
$\mathbb E (\sum_{i\geq1}\tau_i)<\infty.$
 Then,
 $\sum_{i\geq1}\tau_i$ is a.s. finite and for any $t,\eps>0$,
 there exists $N_0\geq1$ such that
 \begin{equation}\label{minoration_proba}
 \p\left(\sum_{i=N_0}^\infty \tau_i\geq \frac{t}{2}\right)\leq \eps.
 \end{equation}
 We fix $t$ and $\eps>0$. For $k\geq N_0$ and $A>0$,
{\setlength\arraycolsep{2pt}
\begin{eqnarray*}
\p_k(X(t)\geq A)&\leq& \p_k(T_{N_0}\geq t/2)+\p_k(X(t)\geq A,T_{N_0}<t/2)\\
&\leq&\p_k\left(\sum_{i=N_0}^k\tau_i\geq t/2\right)+\sup_{s\in[0,t/2]}\p_{N_0}(X(t-s)\geq A)\\
&\leq&\eps+A^{-1}\sup_{s\in[t/2,t]}\Esp_{N_0}(X(s)),
 \end{eqnarray*}}
using  \eqref{minoration_proba} and the Markov inequality in the last inequality.  Making $A$ tend to infinity and recalling  \eqref{moment}, we get
 \eqref{lim_liminf} and the  weak convergence of  $\p_k$ to $\p_\infty$ . 
 Using again \eqref{lim_liminf}  ensures that  for any $t,\eps>0$, there exists $A$ such that
 $\p_\infty(X(t)>A)\leq\eps$. Therefore,
 $$\p_\infty(\inf\{ s\geq 0 : X(s)<+\infty\}>t)\leq\p_\infty(X(t)>A)\leq\eps,$$
so that $ \inf\{ t\geq 0 : X(t)<+\infty\}=0$ $\p_{\infty}$ a.s. It  ends the proof.
 \end{proof}

\section{Behavior of $T_n$ under $\p_\infty$}\label{section_comportement_T_n}
From now on, we assume that the sequences $(\lambda_n)_n$ and $(\mu_n)_n$ satisfy the hypotheses   \eqref{condition_limite},  \eqref{condition_croissance} and \eqref{condition_CDI_mu}. Thus,  according to Propositions \ref{CDI} and \ref{defi_P_infini},  $X$ comes down from infinity and  $\p_\infty$ is well-defined.

\noindent In this section, we study the asymptotic 
 behavior of $T_n$ as $n\to+\infty$ under $\p_\infty$ by establishing a law of of large numbers and a central limit theorem.
Let us note that under $\p_\infty$, 
$T_n=\sum_{i\geq n}\tau_i$, so that
 \eqref{moment_{tn}} yields
$$\Esp_\infty(T_n)=\sum_{i\geq n}\frac{1}{\lambda_{i}\pi_{i}}\sum_{j\geq {i+1}}\pi_j.$$
Then $S<+\infty$ ensures that $\Esp_\infty(T_n)$ decreases to $0$ as $n\to+\infty$.

\medskip \noindent
In the following theorem, we prove that $T_n$ behaves as its mean $\Einf{T_n}$ as $n\to+\infty$. 
 Two regimes appear depending on whether the ratio of  mean times $\Esp_{n+1}(T_{n})/\Esp_\infty (T_n)$ 
vanishes or not. In the first case, the time $T_n$ can be seen as the contribution of independent random variables and a law of large numbers holds. In the second case, the time $T_n$ is essentially given by the sums of $\tau_i$ for $i$ close to $n$ and renormalizing $T_n$ by its mean yields a random limit.
\begin{thm}\label{Behavior_Tn}
We assume that  \eqref{condition_limite}, \eqref{condition_croissance} and \eqref{condition_CDI_mu} hold.
 \begin{enumerate}[{\normalfont (i)}]
  \item
  If $\Esp_{n+1}(T_n)/\Esp_\infty(T_n)\linf{n}  0$, then
 \begin{equation}\label{CV_Tn_proba}
\frac{T_n}{ \Esp_\infty(T_n)}\linf{n} 1 \qquad \text{ in } \  \p_\infty-\text{probability}.
 \end{equation}
Assuming further that $\displaystyle\sum_{n\geq0}\left(\Esp_{n+1}(T_{n})/\Esp_\infty(T_n)\right)^2<+\infty,$ then  \eqref{CV_Tn_proba} holds $\p_\infty$-a.s.
 \item\label{Behavior_Tn_ii}
 If $\Esp_{n+1}(T_{n})/\Esp_\infty(T_n) \linf{n}\alpha$ with $\alpha\in(0,1]$,  then
 $$\frac{T_n}{\Esp_\infty(T_n)}\overset{\mathrm{(d)}}{\underset{n\to+\infty}\longrightarrow}Z:=\sum_{k\geq0}\alpha\left(1-\alpha\right)^kZ_k$$
 where $(Z_k)_k$ is a sequence of i.i.d. random variables whose common Laplace transform $G(a):=\Esp_\infty\left(\exp(-aZ_0)\right)$ is the unique function $[0,+\infty)\to[0,1]$  that satisfies 
 \begin{equation}\label{equation_G}
 \forall a>0, \quad  G(a)\left[l\big(1-G(a(1-\alpha))\big)+1+a(1-l(1-\alpha))\right]=1.
 \end{equation}
  \end{enumerate}
\end{thm}

\bi We refer to  Appendix C for some examples and counterexamples.  For instance, if $\lambda_{k}=k$, then
  $\mu_k= k^{\gamma}\log(k)^{\beta}$ with $\gamma>1$
obeys to the regime (i),  whereas $\mu_k = \exp(\beta k)$ corresponds to the regime (ii). We also stress  that $\Esp_{n+1}(T_{n})/\Esp_\infty(T_n)$ may not converge (Example \ref{ctr}) and that the a.s. convergence can fail under the assumption (i) (Example \ref{contre_exemple}).

\noindent Before proving Theorem \ref{Behavior_Tn}, we need a lemma dealing with the asymptotic behaviors of the first moments of $\tau_n$ as $n\to+\infty$.
\begin{lem}\label{lemme_moments_tau_n}
 Under hypotheses \eqref{condition_limite}, \eqref{condition_croissance} and \eqref{condition_CDI_mu},
 there exist positive constants $C_1,\ C_2,\ C_3$ such that for $n\geq1$
 $$\frac{i!}{\mu^i_{n+1}}\leq \Einf{\tau_n^i}\leq \frac{C_i}{\mu_{n+1}^i},\quad i=1,2,3.$$
 Moreover, under the additional assumption $l=0$, we have
 $$\Einf{\tau_n^i}\underset{n\to+\infty}\sim\frac{i!}{\mu^i_{n+1}},\quad i=1,2,3.$$
\end{lem}

\begin{proof}
By rewriting \eqref{moment_{tn}}, we have
$\Einf{\tau_n}=\sum_{i\geq n}\frac{1}{\mu_{i+1}}\prod_{j=n}^{i-1}\frac{\lambda_{j+1}}{\mu_{j+1}},$ with the convention  $\prod_{j=n}^{n-1}\frac{\lambda_{j+1}}{\mu_{j+1}}=1$. 
Thus, according to Lemma \ref{lemme_technique} applied to $a_i=\lambda_{i+1}/\mu_{i+1}$ and $b_i=1/\mu_{i+1}$, under \eqref{condition_croissance}, we obtain the expected bounds 
for the first moment $(i=1)$. Moreover
\begin{equation}\label{eq:1749}
 \sup_{n,k\geq0}\frac{\Einf{\tau_{n+k}}}{\Einf{\tau_{n}}}<\infty\medskip
\end{equation}
and we can now deal with  the second moment of $\tau_n$. Differentiating \eqref{recurrence_Gn} twice at $a=0$, we get 
 $$\Einf{\tau_{n-1}^2}=\frac{\lambda_n}{\mu_n}\,\Esp(\tau_{n}^2)+2\,\Einf{\tau_{n-1}}^2, \quad n\geq1.$$
Adding that $\left(\Einf{\tau_n^2}\right)_n$ is bounded from point (iv) of Proposition \ref{CDI}, that $\lim_{n\to+\infty}\lambda_n/\mu_n<1$ and that $(\Einf{\tau_n})_n$ satisfies \eqref{eq:1749}, another use of Lemma \ref{lemme_technique} ensures the desired result for $i=2$. Similarly, the case $i=3$ is obtained  by differentiating \eqref{recurrence_Gn} three times.
\end{proof}
\medskip
\begin{proof}[Proof of Theorem \ref{Behavior_Tn}(i)]
We use the notation $$m_n=\Esp_{n+1}(T_{n}),\quad
r_n:=\frac{\Esp_{n+1}(T_{n})}{\Esp_\infty(T_n)}=\frac{m_n}{\Esp_\infty(T_n)}.$$
Assumption (i) means that $r_n\rightarrow 0$. 
Let $\eps>0$. Using Bienaym\'e-Tchebychev inequality and the independence of  the random variables $(\tau_n)_n$, we have
 \begin{equation}\label{Bienayme_Tchebychev}
 \p_\infty\left(\left|\frac{T_n}{\Esp_\infty(T_n)}-1\right|>\eps\right)\leq\frac{\Var(T_n)}{\eps^2 \Esp_\infty(T_n)^2}=\frac{\sum_{k\geq n}\Var(\tau_k)}{\eps^2 \Esp_\infty(T_n)^2}.
 \end{equation}
 As  $\Esp_\infty(T_{n+1})/m_n=1/r_{n}-1\rightarrow + \infty$ as $n\to+\infty$,
 for all $A>0$, there exists an integer $n_0$ such that, for $n\geq n_0$, $\Esp_\infty(T_{n+1})\geq Am_n$ and
 $$\Esp_\infty(T_n)^2=\left(\sum_{k\geq n}m_k\right)^2\geq2\sum_{k\geq n}m_k\sum_{l>k}m_l\geq2A\sum_{k\geq n}m_k^2,$$
 since $\  \sum_{l>k}m_l = \Esp_\infty(T_{k+1}) \geq A m_{k}$.
 Coming back to \eqref{Bienayme_Tchebychev}, for $n\geq n_0$, we have
 \begin{equation}\label{eq:1553}
\p_\infty\left(\left|\frac{T_n}{\Esp_\infty(T_n)}-1\right|>\eps\right)\leq \frac{1}{2A\eps^2}\frac{\displaystyle\sum_{k\geq n}\Var(\tau_k)}{\displaystyle\sum_{k\geq n}m_k^2}.
 \end{equation}
Moreover, according to Lemma \ref{lemme_moments_tau_n}, for $n\geq1$ 
$$\Vinf{\tau_n}\leq\frac{C_2-1}{\mu_n^2}\leq(C_2-1)m_n^2.$$
Hence, the r.h.s. of \eqref{eq:1553} goes to $0$ as $A\to+\infty$  and the proof of the convergence in probability is complete.\\

\noindent  We now prove the a.s. convergence when the series $\sum_nr_n^{2}$ converges. According to the law of large numbers of Proposition 1 in \cite{Klesov83}, we just need to check that
 \begin{equation}\label{condition_serie}
  \sum_{n\geq0}\frac{\Var(\tau_{n})}{\Esp_\infty(T_n)^2}<+\infty.
 \end{equation}

\noindent From the first part of the proof, we know that $\Var(\tau_{n})\leq \widehat C \EE{\tau_{n}}^2$ for some positive constant $\widehat C$.
 So  $\sum_{n\geq1}r_n^{2}<+\infty$ ensures  \eqref{condition_serie} and the proof is complete.
 \end{proof}
 
\noindent  Before proving point (ii) of Theorem \ref{Behavior_Tn}, we prove the following key lemma focusing on the asymptotic behavior of the time $\tau_n$ (recall that we denote its mean by $m_n$). 
 \begin{lem}\label{equivalents_alpha_non_nul}
  If $\lim_{n\to+\infty} r_n=\alpha\in(0,1]$, we have  
 \begin{equation}\label{equivalents_ratios_alpha}
 \lim_{n\to+\infty}\frac{\Esp_\infty(T_{n+1})}{\Esp_\infty(T_n)}=\lim_{n\to+\infty}\frac{m_{n+1}}{m_n}=1-\alpha, \qquad 
\lim_{n\to+\infty}\mu_nm_{n-1}=\frac{1}{1-l(1-\alpha)}.
\end{equation} 
and
 $$\frac{\tau_n}{m_n}\overset{\mathrm{(d)}}{\linf{n}}\zeta$$
 where the Laplace transform of $\zeta$ is the unique solution of  \eqref{equation_G}.  
 \end{lem}

\begin{proof}
We obtain the first part of  \eqref{equivalents_ratios_alpha} by noticing that $\Esp_\infty(T_{n+1})/\Esp_\infty(T_n)=1-r_n$ and that $$\frac{m_{n+1}}{m_n}=\frac{r_{n+1}}{r_{n}}\frac{\Esp_\infty(T_{n+1})}{\Esp_\infty(T_n)}.$$
Moreover,  differentiating \eqref{recurrence_Gn} at $a=0$ yields
\begin{equation*}\label{recurrence_m_n}
1=\frac{\lambda_n}{\mu_n}\frac{m_{n}}{m_{n-1}}+\frac{1}{\mu_n m_{n-1}},
\end{equation*}
which gives the second part of \eqref{equivalents_ratios_alpha} thanks to \eqref{condition_limite}. 
$\newline$

\noindent Let us now prove the uniqueness of  the function satisfying Equation \eqref{equation_G}. 
  For any bounded function $g:[0,+\infty)\to[0,1]$, we define the function $H(g):[0,+\infty)\to[0,1]$ as
 $$H(g):a\longmapsto\frac{1}{1+l(1-g(a(1-\alpha))+a(1-l(1-\alpha))}.$$
 Let $g_1$ and $g_2$ be two solutions of \eqref{equation_G}. We then have $H(g_1)=g_1$, $H(g_2)=g_2$
 and 
 {\setlength\arraycolsep{2pt}
 \begin{eqnarray*}
 \left|g_1(a)-g_2(a)\right| = \left|H(g_1)(a)-H(g_2)(a)\right|  &=& H(g_1)(a)H(g_2)(a) l\left|g_1(a(1-\alpha))-g_2(a(1-\alpha))\right|\\
 &\leq& l\left|g_1(a(1-\alpha))-g_2(a(1-\alpha))\right|
 \end{eqnarray*}}
where we have used that for any $a>0$, $H(g_1)(a)\leq1.$
We then have 
 $\|g_1-g_2\|_\infty\leq l\|g_1-g_2\|_\infty$
 with $l<1$, which entails that $g_1=g_2$ and yields the expected uniqueness. 
$\newline$

\noindent We can now prove the convergence  in distribution of $\tau_n$ as $n\to+\infty$ by a tightness criterion. Indeed, 
for $n\geq0$, let $H_n:[0,+\infty)\longrightarrow[0,1]$ be the function defined as
$$H_n(a)=\Einf{\exp(-a\tau_n/m_n)},\quad a>0.$$
The sequence $(H_n)_n$ is uniformly bounded since $0\leq H_n(a)\leq 1$ for every $n\geq0$ and every $a>0$. Moreover, for $n\geq0$, $H_n$ is differentiable and for $a>0$,
$$\left|H_n'(a)\right|=\Einf{\frac{\tau_n}{m_n}\exp\left(-a\frac{\tau_n}{m_n}\right)}\leq\frac{\Einf{\tau_n}}{m_n}=1.$$
Hence, the family $(H_n)_n$ is equicontinuous since all these functions are $1$-Lipschitzian functions. Then, thanks to Arzel\`a-Ascoli theorem, $(H_n)_{n\geq0}$ is relatively compact.

\noindent We now need to check that $(H_n)_n$ has a unique limit point.  Let us prove that if a subsequence of  $(H_n)$ converges to $H$  uniformly on any compact set of $[0,+\infty)$, then 
$H$ satisfies   \eqref{equation_G} and is then uniquely defined. For that purpose, we use  \eqref{recurrence_Gn}, so  for all  $a>0$ and $n\geq1$, we have
 \begin{equation*}
 G_{n-1}\left(\frac{a}{m_{n-1}}\right)=\left[1+\frac{a}{\mu_n m_{n-1}}+ \frac{\lambda_n}{\mu_n}\left(1-G_{n}\left(\frac{a}{m_{n-1}}\right)\right) \right]^{-1}
  \end{equation*}
that is,
\begin{equation}\label{recurrence_Gn2}
H_{n-1}(a)=\left[1+\frac{a}{\mu_n m_{n-1}}+ \frac{\lambda_n}{\mu_n}\left(1-H_{n}\left(a\frac{m_n}{m_{n-1}}\right)\right) \right]^{-1}.
\end{equation}
According to Lemma \ref{equivalents_alpha_non_nul},  $m_n/m_{n-1}\to1-\alpha$ as $n\to+\infty$.  Thus, if a subsequence  (also denoted by $H_{n}$ for simplicity) converges to $H$ uniformly, we have for every $a>0$
$$\lim_{n\to+\infty}H_{n}\left(a\frac{m_n}{m_{n-1}}\right)=H(a(1-\alpha)).$$
Letting $n\to+\infty$ in \eqref{recurrence_Gn2}, since $\mu_nm_{n-1}\to1/(1-l(1-\alpha)$ and $\lambda_n/\mu_n\to l$, $H$ satisfies \eqref{equation_G}
and  for every $a>0$
$$\lim_{n\to+\infty}\Einf{\exp(-a\tau_n/m_n)}=H(a).$$
Finally, we check   that $H$ is the Laplace transform of some random variable by proving that $H(0^+):=\lim_{a\to0}H(a)=1$. 
 From \eqref{equation_G}, $H(0^+)$ is a solution of   
$lH(0^+)^2-(1+l)H(0^+)+1=0.$
If $l=0$, this equation has only $1$ as a solution. If $l>0$, the two solutions are $1$ and $1/l$. But  $1/l>1$ and  obviously $H(0^+)\leq1$, so  that $1$ is the only  possible  solution. 
Hence, in all cases, $H(0^+)=1$ and that ends the proof.
 \end{proof}

\noindent We can now proceed with the proof of the second part of the theorem.

\begin{proof}[Proof of Theorem \ref{Behavior_Tn} \eqref{Behavior_Tn_ii}]
 Let $Z=\sum_{k\geq0}\alpha(1-\alpha)^kZ_k$ be defined as in the statement of the theorem.
 We use that $T_n=\sum_{k\geq n}\tau_k$ where the $\tau_k$'s are independent and that
  for all  $a_1,a_2,\dots,a_n,b_1,\dots b_n\in [0,1]$, a simple recursion ensures that
  \begin{equation}
   \label{prod}
\left|\prod_{i=1}^na_i-\prod_{i=1}^nb_i\right|\leq\sum_{i=1}^n\left|a_i-b_i\right|.
  \end{equation}
Then,  for every  $a>0$
 {\setlength\arraycolsep{2pt}
 \begin{eqnarray}&&\left|\Esp_\infty\left(\exp\left(-a\frac{T_n}{\Esp_\infty(T_n)}\right)\right)-\Esp_\infty\left(\exp\left(-aZ\right)\right) \right| \nonumber \\
&&\qquad =\left|\prod_{k\geq n}\Esp_\infty\left(\exp\left(-a\frac{\tau_{k}}{ \Esp_\infty(T_n)}\right)\right)-\prod_{k\geq0}\Esp_\infty\left(\exp\left(-a\alpha(1-\alpha)^kZ_k\right)\right) \right|\nonumber\\
  &&\qquad \leq \sum_{k\geq0}\left|\Esp_\infty\left(\exp\left(-a\frac{\tau_{k+n}}{ \Esp_\infty(T_n)}\right)\right)-\Esp_\infty\left(\exp\left(-a\alpha(1-\alpha)^kZ_k\right)\right) \right|. \label{eq:1816}
   \end{eqnarray}}
From Lemma \ref{equivalents_alpha_non_nul}, we know that in
$\p_\infty$-distribution, 
$\tau_n/ m_n$ converges to $\zeta$. 
Then, thanks to \eqref{equivalents_ratios_alpha} and the fact that $r_n\to\alpha$, we have for $k\geq0$
 $$\frac{\tau_{k+n}}{\Esp_\infty(T_n)}=\frac{m_{n+k}}{\Esp_\infty(T_{n+k})}\prod_{i=1}^k\frac{\Esp_\infty[T_{n+i}]}{ \Esp_\infty[T_{n+i-1}]}\cdot\frac{\tau_{k+n}}{m_{n+k}}\overset{\textrm{(d)}}{\linf{n}}\alpha(1-\alpha)^k\mathcal \zeta.$$
 The uniqueness in \eqref{equation_G} ensures that the variables $(Z_k)_k$ are distributed as $\zeta$. Then, with the last display, we get that all the terms of the sum in \eqref{eq:1816} vanish as $n\to+\infty$.
We proceed by bounded convergence. Using that
$1-\exp(-x)\leq x$ for any $x\geq0$,  we get for $k,n\geq0$
  {\setlength\arraycolsep{2pt}
 \begin{eqnarray}
&&  \left|\Esp_\infty\left(\exp\left(-a\frac{\tau_{k+n}}{ \Esp_\infty(T_n)}\right)\right)-\Esp_\infty\left(\exp\left(-a\alpha(1-\alpha)^kZ_k\right)\right) \right|\nonumber \\
&& \qquad \leq 
  \left|1-\Esp_\infty\left(\exp\left(-a\frac{\tau_{k+n}}{ \Esp_\infty(T_n)}\right)\right)\right|+\left|1-\Esp_\infty\left(\exp\left({-a\alpha(1-\alpha)^kZ_k}\right)\right) \right|\nonumber\\
  &&\qquad \leq a\frac{m_{k+n}}{ \Esp_\infty(T_n)}+a\alpha(1-\alpha)^k\Esp_\infty(Z_0).\label{1911}
 \end{eqnarray}}
By differentiating \eqref{equation_G} at $0$, one finds 
$\Esp_\infty[Z_0]=1$. Moreover,
$$\frac{m_{n+k}}{ \Esp_\infty(T_n)}=\frac{ \Esp_\infty(T_{n+1})}{ \Esp_\infty(T_n)}\frac{ \Esp_\infty(T_{n+2})}{ \Esp_\infty(T_{n+1})}\cdots\frac{\Esp_\infty(T_{n+k})}{ \Esp_\infty(T_{n+k-1})}\frac{m_{n+k}}{\Esp_\infty(T_{n+k})}.$$
Since $m_{k+n}/\Esp_\infty[T_{k+n}]\leq 1$ and $\Esp_\infty(T_{n+1})/ \Esp_\infty(T_n)\to1-\alpha<1$ as $n\to+\infty$, there exist  $n_0\in\Nat$, $\beta<1$ and $C>0$ such that
$m_{k+n}/\Esp_\infty(T_n)\leq C\beta^k$ for all $k\geq0, n\geq n_0.$
 Thus, coming back to \eqref{1911}, for $n\geq n_0$, we have 
 $$\left|\Esp_\infty\left(\exp\left(-a\frac{\tau_{k+n}}{ \Esp_\infty(T_n)}\right)\right)-
\Esp_\infty\left(\exp\left(-a\alpha(1-\alpha)^kZ_k\right)\right)\right|\leq C\beta^k+ a\alpha(1-\alpha)^k.$$
 Since the r.h.s. in the last display is summable, the proof is complete.
\end{proof}

\bigskip \noindent 
We end this section by giving a central limit theorem (C.L.T.) satisfied by the sequence $(T_n)_n$.
 \begin{thm}\label{TLC_Tn}
 We suppose that assumptions \eqref{condition_limite}, \eqref{condition_croissance} and \eqref{condition_CDI_mu} hold. Then, if 
  \begin{equation}\label{hypothese_variances}
   \lim_{n\to+\infty}\frac{\Vinf{\tau_n}}{\Var_\infty(T_n)}=0
  \end{equation}
and if
\begin{equation}\label{hypothese_moment_ordre3}
   \lim_{n\to+\infty}\Vinf{T_n}^{-3/2}\sum_{k\geq n}\Einf{\left|\tau_{k}-\Einf{\tau_k}\right|^3}=0,
  \end{equation}
we have 
\begin{equation*}\label{CLT_T_n}
 \frac{T_n-\Einf{T_n}}{\Var_\infty(T_n)^{1/2}}\underset{n\to+\infty}{\overset{\mathrm{(d)}}\longrightarrow} N,
\end{equation*}
where $N$ follows a standard normal distribution.
 \end{thm}
\noindent Notice that by applying Lemma \ref{lemme_moments_tau_n} and by using assumption \eqref{condition_croissance}, there is $C>0$ such that
$\frac{\Einf{\tau_n}}{\Einf{T_n}}\leq C\frac{\Vinf{\tau_n}}{\Vinf{T_n}}.$ 
Thus,  hypothesis \eqref{hypothese_variances} implies that we are in the regime (i) of Theorem \ref{Behavior_Tn}.  We refer to the first example in Appendix.
 
\begin{proof}[Proof of Theorem \ref{TLC_Tn}]
First, Lemma \ref{lemme_moments_tau_n} gives that  for every $n\geq0$,
 $ \Vinf{\tau_n} \leq \frac{C_{2}-1}{\mu_{n+1}^2 }$. Recalling
 Assumption \eqref{condition_croissance}, we get for $k,n\geq0$
  $$\Vinf{\tau_{k+n}}\leq \frac{C_{2}-1}{\mu_{n+k+1}^2}\leq \frac{(C_{2}-1)\,K^2}{\mu_{n+1}^2}\leq (C_{2}-1)\,K^2\,\Vinf{\tau_n}$$
  where $K=\sup_{k,n\geq1}\mu_n/\mu_{n+k}$. 
  Thus assumption \eqref{hypothese_variances} entails the uniform convergence
  \begin{equation}
   \label{uniformite}
\sup_{k\geq0}\frac{\Vinf{\tau_{k+n}}}{\Vinf{T_n}}\linf{n}0.
  \end{equation}
Let us now prove that
$$Z_n:=\frac{T_n-\Einf{T_n}}{\Vinf{T_n}^{1/2}}$$ converges in distribution as $n\to+\infty$ toward a standard normal random variable.
We follow ideas of the proof of Theorem 27.2 in \cite{Billingsley1986} where Billingsley establishes a central limit theorem for partial sums of independent random variables thanks to L\'evy theorem.
Let $t$ be a fixed real number.  
 We note that by $(\ref{uniformite})$, 
$$\sum_{k\geq0}\log \left(1-\frac{t^2}{2}\frac{\Vinf{\tau_{k+n}}}{\Vinf{T_n}}\right)\equinf{n}-\frac{t^2}{2}\sum_{k\geq0}\frac{\Vinf{\tau_{k+n}}}{\Vinf{T_n}}=-\frac{t^2}{2},$$
so we just need to prove that
$$U_n:=\Einf{\exp\left(itZ_n\right)}-\prod_{k\geq 0}\left(1-\frac{t^2}{2}\frac{\Vinf{\tau_{k+n}}}{\Vinf{T_n}}\right)$$
vanishes as $n\to+\infty$ to conclude. 
  First, since the $\tau_n$'s are independent,  for all $t\in\R, n\geq0$
  \begin{equation}\label{eq1711}
   |U_n|=
   \left|\prod_{k\geq0}\Einf{\exp\left(it\frac{\tau_{k+n}-\Einf{\tau_{k+n}}}{\Vinf{T_n}^{1/2}}\right)}-\prod_{k\geq 0}\left(1-\frac{t^2}{2}\frac{\Vinf{\tau_{k+n}}}{\Vinf{T_n}}\right)\right|.
  \end{equation} 
  According to  \eqref{uniformite}, for $n$ large enough and for any $k$,
  all the factors of the second product of \eqref{eq1711} are less than 1. Hence, thanks to \eqref{prod}, we have the 
  inequality
  \begin{equation}\label{eq1727}
   |U_n|\leq \sum_{k\geq0}\left|\Einf{\exp\left(it\frac{\tau_{k+n}-\Einf{\tau_{k+n}}}{\Vinf{T_n}^{1/2}}\right)}-1+\frac{t^2}{2}\frac{\Vinf{\tau_{k+n}}}{\Vinf{T_n}}\right|.
  \end{equation}
  According to equation (27.11) in \cite[p.369]{Billingsley1986}, for any centered random variable $\xi$ with a finite second moment, we have
  $\left|\EE{\exp(it\xi)}-1+\Var(\xi)t^2/2\right|\leq\EE{\min(|t\xi|^2,|t\xi|^3)},\quad t\geq 0.$
 Using this inequality with the random variables $\tau_{n+k}-\Einf{\tau_{n+k}}$, we obtain from \eqref{eq1727} that
 \begin{equation*}\label{eq1745}
   |U_n|\leq |t|^3\sum_{k\geq0}\frac{\Einf{\left|\tau_{k+n}-\Einf{\tau_{k+n}}\right|^3}}{\Vinf{T_n}^{3/2}}.
  \end{equation*}
  and using assumption \eqref{hypothese_moment_ordre3}, $U_n$ goes to $0$ as $n\to+\infty$.
Is completes the proof.
 \end{proof}

\section{Behavior of $X(t)$ as $t$ goes to $0$}\label{section_comportement_X}
\label{SCDI}
From the results of  Section \ref{section_comportement_T_n}, we can describe the behavior of $X$ for small times, when it starts at $+\infty$. 
\subsection{Law of large numbers}
We first prove that under $\p_\infty$, $X(t)$ behaves as $v(t)$ as $t\to0$ where
\begin{equation}\label{defi_v}
v(t):=\inf\{n\geq0; \Esp_\infty(T_n)\leq t\}
\end{equation}
is   the generalized inverse function of 
$n\mapsto \Esp_\infty(T_n)= \frac{1}{\lambda_{n}\pi_{n}}\sum_{i\geq n+1}\pi_i.$

\medskip \noindent The function $v$  is a non-increasing  function which tends to infinity when $t$ tends to $0$. 

\medskip \noindent Two asymptotic behaviors appear, which are inherited 
from  Theorem \ref{Behavior_Tn}. First, we  assume  that $\lambda_n/\mu_n\rightarrow 0$
as $n\rightarrow \infty$ and that the death rate regularly varies in the neighborhood of $+\infty$, see Section \ref{fonctions_variation_reguliere} for details.
Indeed, it ensures that the a.s. convergence of Theorem \ref{Behavior_Tn}(i)  holds and these assumptions are then
 essential  to derive the behavior of $X$ from that $(T_n : n\in \mathbb N)$.

\begin{thm}
 \begin{enumerate}[{\normalfont(i)}]
 \item If $\lim_{n\to+\infty}\lambda_n/\mu_n=0$ and $(\mu_n)_n$ regularly varies with index $\rho>1$, we have 
 \begin{equation*}\label{equivalent_Xt_2}
  \lim_{t\to0}\frac{X(t)}{v(t)}=1 \qquad \p_\infty-\text{a.s}.
 \end{equation*}
 
 \item Under assumptions \eqref{condition_limite}, \eqref{condition_croissance} and \eqref{condition_CDI_mu} and if $\mathbb E_{n+1}(T_n)/\Einf{T_n}\to \alpha>0$ as $n\to+\infty$,
 then 
 \begin{equation*}
  \lim_{t\to0}\frac{X(t)}{v(t)}=1 \qquad \text{in }\p_\infty-\text{probability}.
 \end{equation*}
 \end{enumerate}\label{Theoprincipal}
  \end{thm}
\noindent 
We refer to Appendix for some examples.
Further, we  remark that if  $\lambda_n=0$ and $\mu_n=n(n-1)/2$, $X(t)$ is the number of blocks of the  Kingman coalescent at time $t$. We recover here from (i) the speed of coming down from infinity obtained for these processes by Aldous in paragraph 4.2. of \cite{Aldous99}:  $tX(t)\underset{t\to 0}{\longrightarrow}2\as.$

 \noindent The extension to  the general case of $\Lambda$-coalescent has been solved   by Berestycki, Berestycki and Limic \cite{Berestiycki_Limic}, but it is not directly included in our work for simultaneous deaths. 
\begin{proof}[Proof of Theorem \ref{Theoprincipal}(i)] First, we notice that the hypotheses of point (i) imply that assumptions \eqref{condition_limite}, \eqref{condition_croissance} and \eqref{condition_CDI_mu} with $l=0$ are all satisfied.
We now prove   that $\sum_{n}\left( \mathbb E_{n+1}(T_n)/\mathbb E_{\infty}(T_n)\right)^2<+\infty$ to get  the a.s. convergence from  Theorem \ref{Behavior_Tn}(i).

\noindent Since $l=0$ and according to Lemma \ref{lemme_moments_tau_n}, $m_n\sim_{n\to +\infty}1/\mu_{n+1}$, which implies that $(m_n)_n$ regularly varies at $+\infty$ with index $-\rho<-1$.  %
Then, according to Lemma \ref{VR_restes_sommes} in Appendix applied to $(1/\mu_n)_n$, we have
$$\frac{\mathbb E_{n+1}(T_n)}{\mathbb E_{\infty}(T_n)}\equinf{n}\left(\mu_{n+1}\sum_{k\geq n+1}\frac{1}{ \mu_k}\right)^{-1}\equinf{n}\frac{\rho-1}{n+1},$$
which entails that
 $\sum_n \left(\frac{\mathbb E_{n+1}(T_n)}{\mathbb E_{\infty}(T_n)}\right)^2<+\infty$. So  
 Theorem \ref{Behavior_Tn}(i) yields
\begin{equation}\label{eq:1822}
\frac{T_n}{\Esp_\infty(T_n)}\linf{n}1\qquad \p_\infty-\textrm{a.s.}
\end{equation}
\medskip
\noindent The proof  is now organized as follows: firstly we consider the a.s. non-increasing process $Y$ defined by
$$Y(t)=n \quad \text{if} \quad t\in [T_{n},T_{n-1})$$
and prove  that this (more regular) process comes down from infinity at speed $v(t)$.
Secondly, we compare the process  $X(t)$  to $Y(t)$ as $t\rightarrow 0$ to get the result.

\noindent 
Thanks to Proposition \ref{VR_inverse}, $v$ regularly varies at $0$ with index $1/(1-\rho)$
and $v(\Esp_\infty(T_n))\sim n$ as $n\to+\infty$. 
Thus, from \eqref{eq:1822} and Lemma \ref{lemme_equivalents} we obtain that almost surely
$$v(T_n)\equinf{n}v(\Esp_\infty(T_n))\equinf{n}n.$$
Adding that  $v$ is non-increasing, we get  a.s. that for every $\eps>0$, there exists $n_0 \in \mathbb N$ such that for $n\geq n_0$
$$1-\eps\leq \frac{n}{v(T_{n})}\leq\frac{n}{v(T_{n-1})}\leq 1+\eps.$$
Let $t< T_{n_0}$ so that $Y(t)>n_0$, then if $\,t\in [T_{n}, , T_{n-1})$,  $$1-\eps \leq \frac{n}{v(T_{n})} \leq  \frac{Y(t)}{v(t)}\leq \frac{n}{v(T_{n-1})}\leq 1+\eps.$$
That ensures 
\begin{equation}
\label{compY}
 \lim_{t\rightarrow 0} \frac{Y(t)}{v(t)}=1 \as
\end{equation}
Let us now check that $X(t)\sim Y(t)$ as $t\rightarrow 0$ by proving that the  heights of the excursions of $X$ between $T_{n}$ and $T_{n-1}$  are negligible compared to $n$. For that purpose, we introduce
the number of birth events between the times $T_n$ and $T_{n-1}$:
$$H_n:=\#\{  s  \in [T_n,T_{n-1}) : X(s)-X(s-)>0\}, \quad n\geq1.$$
For any $t \in [T_n,T_{n-1})$, $Y(t)=n$ and
$ 0\leq X(t)-Y(t) \leq H_n$, so
\begin{equation}\label{encadrement_X_Y}
0\leq\frac{X(t)}{v(t)}-\frac{Y(t)}{v(t)}\leq \frac{H_{Y(t)}}{v(t)}=\frac{H_{Y(t)}}{Y(t)} \frac{Y(t)}{v(t)}.
 \end{equation}
Using (\ref{compY}), we just need to prove that
${H_n}/{n}\rightarrow 0$ a.s. as $n\rightarrow \infty$ to conclude that $X(t)/v(t)\rightarrow 1$ as $t\rightarrow 0$.

\noindent For that purpose, we consider $\widehat G_n(a)=\Einf{\exp(-aH_n)}$ the Laplace transform of $H_n$. In the same vein as we have obtained \eqref{recurrence_Gn}, by applying the strong Markov property at the first time when $X$ jumps after $T_n$, we have the recursion formula
\begin{equation}\label{recurrence_hat_Gn}
\widehat G_n(a)=\frac{\mu_n}{\lambda_n+\mu_n}+\frac{\lambda_n}{\lambda_n+\mu_n}e^{-a}\widehat G_n(a)\widehat G_{n+1}(a),\quad a\geq 0,\ n\geq1.
\end{equation}
Differentiating \eqref{recurrence_hat_Gn} twice at $a=0$, the second moment of $H_n$ satisfies the following recursion formula
\begin{equation}
\frac{\mu_n}{\lambda_n}\Einf{H_n^2}=\Einf{H_{n+1}^2}+1+2\left(\Einf{H_n}
+\Einf{H_{n+1}}+\Einf{H_n}
\Einf{H_{n+1}}\right).\label{recurrence_Hn2}
\end{equation}
Let us prove that the right hand side of the latter is uniformly bounded in $n\geq0$. Notice that $H_n$ equals the number of positive jumps between time $T_n$ and $T_{n-1}$ of a random walk whose transition probabilities are given by
$p_{i,i+1}=\lambda_i/(\lambda_i+\mu_i)$, $ p_{i,i-1}=\mu_i/(\lambda_i+\mu_i)$ for $i\geq 1.$
Since $\lambda_n/\mu_n$ vanishes as $n\to+\infty$, one can choose $n_0$ large enough so that
$$p:=\sup_{n\geq n_0}\lambda_n/(\lambda_n+\mu_n) <1/2.$$
Moreover, for $n\geq n_0$, $H_n$ is stochastically dominated by $T$, the hitting time of $n-1$ by a simple random walk starting at $n$, with probability transitions $(1-p,p)$.
Thus, 
$\sup_{n\geq n_0}\Einf{H_n^2}\leq \Einf{T^2}<\infty$ because $p<1/2$ and the sequences $(\Einf{H_n})_n$ and $(\Einf{H_n^2})_n$ are bounded. It entails that the right hand side of \eqref{recurrence_Hn2} is bounded and
there is $C>0$ such that
\begin{equation}\label{majoration_moment2_Hn}
 \Einf{H_n^2}\leq C\frac{\lambda_n}{\mu_n},\quad n\geq 1.
\end{equation}
Finally, using that $l=0$,
$\Einf{\sum_{n\geq1}\left(\frac{H_n}{n}\right)^2}\leq C\sum_{n\geq1}\frac{1}{n^2}\frac{\lambda_n}{\mu_n}<\infty.$
In particular, $H_n/n$ almost surely goes to $0$ as $n\to+\infty$ and we get the expected convergence. 
\end{proof}
 \begin{proof}[Proof of Theorem \ref{Theoprincipal}(ii)]
From Theorem \ref{Behavior_Tn}(ii), we know that under $\p_\infty$,
$$\frac{T_n}{\Esp_\infty(T_n)}\overset{\textrm{(d)}}{\linf{n}}Z=\sum_{k\geq0}\alpha\left(1-\alpha\right)^kZ_k$$
with $\alpha\in (0,1]$ and where $(Z_k)_k$ is a sequence of i.i.d. random variables whose Laplace transform satisfies \eqref{equation_G}. From this equation, one deduces 
 $\Esp_\infty[Z_0]=1$, which implies $\Esp_{\infty}(Z)=1<\infty$. In particular, $\p_\infty(Z<+\infty)=1.$
Furthermore, using again \eqref{equation_G}, $$\p_\infty(Z=0)\leq\p_\infty(Z_0=0)=\lim_{a\to+\infty}G(a)=0.$$
Hence, for any $\epsilon >0$, there exist $0<A\leq B$ such that $\p_\infty( Z \in [A,B])\geq 1- \epsilon /2$ and for $n$ large enough,
\begin{equation}\label{eq_1520}
\p_\infty(A\leq  T_n/ \Esp_\infty(T_n)  \leq B)\geq 1-\epsilon.
\end{equation}
Moreover, according to \eqref{equivalents_ratios_alpha}, for $N\geq0$,
\begin{equation*}\label{eq:1632}
\frac{\Esp_\infty[T_{n+N}]}{\Esp_\infty(T_n)}\linf{n}(1-\alpha)^N.
 \end{equation*}
Since $0<\alpha\leq1$, there exists $N_0$ such that for all $n,N\geq N_0$,
\begin{equation}\label{eq_1521}
 \Esp_\infty[T_{n+N}]/ \Esp_\infty(T_n)\leq \min(1/(2B),A/2).
 \end{equation}
By the definition \eqref{defi_v} of the function $v$, we have
$$\Esp_\infty\left(T_{v(t)}\right)\leq t<\Esp_\infty\left(T_{v(t)-1}\right).$$
It implies that for any $t>0$ and $N\geq1$,
$$\pinf{T_{v(t)+N}\leq \frac{t}{2}}\geq \pinf{T_{v(t)+N} \leq  \frac{\Esp_\infty\left(T_{v(t)}\right)}{2}}=\pinf{\frac{T_{v(t)+N}}{ \Esp_\infty\left(T_{v(t)+N}\right)} \leq  \frac{ \Esp_\infty\left(T_{v(t)}\right)}{ 2\Esp_\infty\left(T_{v(t)+N}\right)}}.$$
Hence, using \eqref{eq_1520} and \eqref{eq_1521}, there is $N\geq1$ such that for $t$ small enough  
$$\p_\infty(T_{v(t)+N}\leq t/2)\geq\p_\infty \left(\frac{T_{v(t)+N}}{ \Esp_\infty\left(T_{v(t)+N}\right)}   \leq B  \right )\geq 1-\epsilon.$$
We similarly get that for $t$ small enough
$$\p_\infty(T_{v(t)-N}\geq 2t)\geq\p_\infty \left(\frac{T_{v(t)-N}}{ \Esp_\infty\left(T_{v(t)-N}\right)}   \geq A \right )\geq 1-\epsilon.$$
Then, we have
$$\p_\infty(T_{v(t)+N}\leq t/2, T_{v(t)-N}\geq 2t)\geq 1-2\epsilon.$$
It means that
$\p_\infty(   X(t) \in [v(t)-N,v(t)+N])\geq 1-2\epsilon$
and ensures that
$X(t)/v(t)$ tends to $1$ in probability as $t\rightarrow 0$.
\end{proof}

\subsection{Central limit theorem}
We have proved that $X$ satisfies a strong law of large numbers when $l=0$ and $(\mu_n)_n$ regularly varies. 
Under a little stronger assumption, we are now giving a central limit theorem (C.L.T.). 

\begin{thm}\label{TLC_X}
 If $\lim_{n\to+\infty}\frac{\lambda_n}{\mu_n}=0$, $\sum_{n\geq1}\frac{1}{n}\frac{\lambda_n}{\mu_{n}}<+\infty$ and $(\mu_n)_n$ regularly varies with index $\rho>1$, then
\begin{equation}\label{eq:1844}
\sqrt{(2\rho-1)v(t)}\left(\frac{X(t)}{v(t)}-1\right)\underset{t\to0}{\overset{\mathrm{(d)}}\longrightarrow} N,
\end{equation}
where $N$ follows a standard normal distribution.
\end{thm}

\begin{proof}
 We first prove that under the assumptions of Theorem \ref{TLC_X}, $T_n$ satisfies the C.L.T. stated in Theorem \ref{TLC_Tn}. We have already shown at the beginning of the proof of Theorem \ref{Theoprincipal}(i) that assumptions \eqref{condition_limite}, \eqref{condition_croissance} and \eqref{condition_CDI_mu} hold if 
 $l=0$ and $(\mu_n)_n$ regularly varies.
  It then remains to check that \eqref{hypothese_variances} and \eqref{hypothese_moment_ordre3} are also satisfied.  
  From Lemma \ref{lemme_moments_tau_n}, $\Vinf{\tau_n}\sim 1/\mu_{n+1}^2$ as $n\to+\infty$, which implies that 
   $(\Vinf{\tau_n})_n$ regularly varies with index $-2\rho$.  Then  the fact that $\Vinf{T_n}=\sum_{i=n+1}^{\infty} \Vinf{\tau_n}$ and Lemma \ref{VR_restes_sommes} ensure that $\Vinf{T_n}$ regularly varies with index $1-2\rho$ and we get
  \be
  \label{varTn}
  \Vinf{T_n} \sim \frac{1}{(1-2\rho)\mu_{n+1}^2}.
  \ee
  Therefore we have
  $\,\frac{\Vinf{\tau_n}}{\Var_\infty(T_n)}\equinf{n}\frac{1-2\rho}{n}$,
  which entails \eqref{hypothese_variances}.  
\noindent
By the triangular inequality and the binomial theorem, we have 
$$\Einf{\left|\tau_{n}-\Einf{\tau_n}\right|^3}\leq\Einf{\tau_n^3}+3m_n\Einf{\tau_n^2}+4m_n^3.$$
Thanks to Lemma \ref{lemme_moments_tau_n}, all the terms of the r.h.s. are of order of magnitude $1/\mu_{n+1}^3$ as $n\to+\infty$.  
  Thus, using  again Lemma \ref{VR_restes_sommes} and \eqref{varTn}, there is a positive constant $C'$ such that
  $$ \frac{\sum_{k\geq n}\Einf{\left|\tau_{k}-\Einf{\tau_k}\right|^3}}{\Vinf{T_n}^{3/2}}\leq C'\,\frac{2\rho-1}{(3\rho -1) \sqrt{n}}$$
and vanishes as $n\to+\infty$. So \eqref{hypothese_moment_ordre3}
holds  and $T_n$ satisfies the following C.L.T. 
\begin{equation}\label{TCL_reecriture}
 \widetilde Z_n:=\frac{T_n-\Einf{T_n}}{\sqrt{\Vinf{T_n}}}\overset{\textrm{(d)}}{\linf{n}} N.
\end{equation}

\noindent We now prove that $X$ satisfies the C.L.T. \eqref{eq:1844}. To do so, we first establish a C.L.T. for the process $Y$ where we recall from the proof of Theorem \ref{Theoprincipal}(i) that $Y$ denotes the a.s. non-increasing process defined as $Y(t)=n$ if $t\in[T_n,T_{n-1})$. 
This process is more tractable than $X$ and we will derive the C.L.T. for $X$ from that of $Y$.\\
As  $Y$ is non-increasing, we can follow the proof of C.L.T for   renewal processes (as suggested by Aldous for Kingman's coalescent, cf. \cite{Aldous99}). More precisely, for any $t>0, x\in\R$, we have
 $$\pinf{\sqrt{v(t)}\left(\frac{Y(t)}{v(t)}-1\right)\geq x}=\pinf{Y(t)\geq s_x(t)}=\pinf{T_{s_x(t)}\geq t}$$ 
 where we denote $s_x(t):=\lfloor v(t)+x\sqrt{v(t)}\rfloor$ ($\lfloor\cdot\rfloor$ is the floor function). We then have
 $$\pinf{\sqrt{v(t)}\left(\frac{Y(t)}{v(t)}-1\right)\geq x}=
 \pinf{\widetilde Z_{ s_x(t)}\geq \frac{t-\Einf{T_{s_x(t)}}}{\sqrt{\Vinf{T_{s_x(t)}}}}}.$$
Using \eqref{TCL_reecriture}, we just need to prove that 
 \begin{equation}\label{eq:1747}
 \frac{t-\Einf{T_{s_x(t)}}}{\sqrt{\Vinf{T_{s_x(t)}}}}\underset{t\to0}\longrightarrow x\sqrt{2\rho-1}
 \end{equation}
to obtain the expected C.L.T. for $Y$. 
From the definition \eqref{defi_v} of the function $v$, we have
\begin{equation}\label{encadrement_1515}
\Einf{T_{v(t)}}-\Einf{T_{s_x(t)}}\leq t-\Einf{T_{s_x(t)}}\leq \Einf{T_{v(t)-1}}-\Einf{T_{s_x(t)}}.
\end{equation}
Let us first deal with the l.h.s. and write
$$\Einf{T_n}-\Einf{T_{\lfloor n+x\sqrt n\rfloor}}=\sum_{k=n}^{\lfloor n+x\sqrt n\rfloor-1}m_k\equinf{n}x\sqrt{n}m_n.$$
Indeed,%
\begin{eqnarray*}
 \left|\sum_{k=n}^{\lfloor n+x\sqrt n\rfloor-1}m_k-(\lfloor n+x\sqrt n\rfloor-n)m_n\right| &\leq& m_n\sum_{k=n}^{\lfloor n+x\sqrt n\rfloor-1}\left|\frac{m_k}{m_n}-1\right|\\
 &\leq&x\sqrt{n}m_n\sup_{u\in[0,x]}\left|\frac{m_{\lfloor n+u\sqrt n\rfloor}}{m_n}-1\right|,
\end{eqnarray*}
and  the second part of Lemma \ref{lemme_equivalents} ensures that the latter supremum vanishes as $n\to+\infty$ since 
 $m_n$ regularly varies.  \\
Then, by successively applying Lemma \ref{lemme_equivalents} with $f(y)=y, g(y)=\lfloor y+u\sqrt y\rfloor$ and $h(n)=\Vinf{T_n}$, Lemma \ref{lemme_moments_tau_n}, Lemma \ref{VR_restes_sommes}(i) and \eqref{varTn}, we have the equivalences
\begin{equation*}
 \frac{\Einf{T_n}-\Einf{T_{\lfloor n+x\sqrt n\rfloor}}}{\sqrt{\Vinf{T_{\lfloor n+x\sqrt n\rfloor}}}}\equinf{n}\frac{x\sqrt{n}m_n}{\sqrt{\Vinf{T_n}}}\equinf{n}x\sqrt{2\rho-1}.
\end{equation*}
Following the same steps for the 
 r.h.s of \eqref{encadrement_1515} ensures that \eqref{eq:1747} holds. 
\medskip \noindent We end the proof by deducing \eqref{eq:1844} from the C.L.T. satisfied by $Y$. 
 Indeed, since 
 $$\sqrt{v(t)}\left(\frac{X(t)}{v(t)}-1\right)=\sqrt{v(t)}\left(\frac{Y(t)}{v(t)}-1\right)+\frac{X(t)-Y(t)}{\sqrt{v(t)}},$$
 it is now sufficient to show that the second term of the latter goes to $0$ in probability as $t\to0$.
 Keeping the same notation as in the proof of Theorem \ref{Theoprincipal}(i), from \eqref{encadrement_X_Y}, we almost surely have
 \begin{equation}\label{encadrement_X_Y_2}
 0\leq \frac{X(t)-Y(t)}{\sqrt{v(t)}}\leq \frac{H_{Y(t)}}{\sqrt{Y(t)}}\frac{\sqrt{Y(t)}}{\sqrt{v(t)}}.
 \end{equation}
From \eqref{majoration_moment2_Hn}, there is $C$ such that
$$\Einf{\sum_{n\geq1}\left(\frac{H_n}{\sqrt{n}}\right)^2}\leq C\sum_{n\geq1}\frac{1}{n}\frac{\lambda_n}{\mu_n}.$$
Since this series converges by hypothesis, $H_n/\sqrt{n}$ almost surely goes to $0$ as $n\to+\infty$.
Hence, since we also have $Y(t)\sim v(t)$ as $t\to0$ with probability 1, the r.h.s. of \eqref{encadrement_X_Y_2} vanishes as $t\to0$, which ends up the proof. 
 \end{proof}

\section{Tail distribution at 0 of the extinction time}\label{section_fast_extinction}
In the following result, we focus on  the probability that the extinction  of the process $X$  occurs for small times.
\begin{thm} 
\begin{enumerate}[{\normalfont(i)}]
 \item If for every $n\geq0$, $\lambda_n=0$ (pure death case) and if $(\mu_n)$ regularly varies at $+\infty$ with index $\rho>1$, then 
 $$t\longmapsto- \log \p_\infty(T_0\leq t)$$ regularly varies at $0$ with index $1/(1-\rho)$.
\item If $\lim_{n\to+\infty}\frac{\lambda_n}{\mu_n}=0$ and $(\mu_n)_n$ regularly varies with index $\rho>1$,
$$ \frac{\log \left( - \log \p_{\infty}(T_0\leq t)\right)}{\log t } \underset{t\to0}\longrightarrow\frac{1}{1-\rho} .$$
 \end{enumerate}\label{descente_rapide}
  \end{thm}

\noindent In the pure death case (i), the time of extinction is the sum of independent exponential random variables. That allows us to get an explicit expression of its Laplace exponent and the result comes from a Tauberian theorem
(Lemma \ref{resultat_Bingham}). It's a key point where we needed the regular variation of $(\mu_{n})_{n}$. To prove (ii), we first use the speed of coming down from infinity  obtained in Theorem \ref{Theoprincipal} on time interval $[0,t/2]$. Then we compare the trajectory with
the pure death case via a coupling argument and conclude thanks to (i). 
\begin{proof}[Proof of Theorem \ref{descente_rapide}(i)]
Let $\phi$ denote the Laplace transform of $T_0$
$$\phi(a):=\Esp_\infty\left[ \exp(-aT_0)\right],\quad a>0.$$
Let us prove that $\,- \log \phi$ regularly varies with index $1/\rho$ at $+\infty$. Lemma \ref{resultat_Bingham} will thus imply the result.

\medskip \noindent 
In the pure death case, the times $\tau_i$ are independent exponential random variables with respective parameters $\mu_{i+1}$. Then, for $a>0$,
$$\phi(a)=\prod_{i\geq 0} \Esp\left[\exp(-a \tau_i)\right]=\prod_{i\geq 1} \frac{1}{1+a/\mu_i}.$$
Using that $\log(1+a)=a(1+r(a))$ with $r$ satisfying $\lim_{a\to0}r(a)=0$, we write 
\begin{equation*}\label{taylor}
- \log\phi(a)= \sum_{ i : \mu_i \leq a\log a} \log(1+a/\mu_i) + \left(1+R(a)\right)\sum_{ i : \mu_i \geq a\log a} a/\mu_i,
 \end{equation*}
where
$$0\leq R(a)\leq \sup_{y\in[0,1/\log a]}r(y)\linf{a}0.$$
Moreover, since $(\mu_i)_i$ regularly varies with index $\rho$, we know (cf. \cite[Thm 1.5.3]{Bingham1989}) that the increasing sequence $(\inf_{n\geq i}\mu_n)_{i}$ regularly varies with index $\rho$. Therefore, according to Proposition \ref{VR_inverse}, the application
$$a\longmapsto i_a:=\min\{ i \in \mathbb N :   \inf_{n\geq i}\mu_n \geq a \log a\}$$
is regularly varying at $\infty$ with index $1/\rho$.
Then,
$$h_1(a):=\sum_{i : \mu_i \geq  a \log a} \frac{a}{\mu_i}=a\sum_{i\geq i_a} \frac{1}{\mu_i}$$
is regularly varying at $\infty$ with index $1+(1-\rho)/\rho=\rho$ since $\sum_{i\geq i_a}1/\mu_i$ is the composition of the two functions $a\mapsto i_a$ and $n\mapsto\sum_{i\geq n} 1/\mu_i$, which both regularly vary with respective indices $1/\rho$ and $1-\rho$. 

\noindent Furthermore, for $a\geq 1$, we have
$$
h_2(a)=\sum_{i\leq i_a}\log (1+a/\mu_i)=\sum_{i\leq i_a}\left(\log a+ \log\left(\frac{1}{a}+\frac{1}{\mu_i}\right)\right)=
i_a\log a+\sum_{i\leq i_a}\log\left(\frac{1}{a}+\frac{1}{\mu_i}\right).
$$
The second term in the r.h.s. is less than $\sum_{i\geq1}\log\left(1+1/\mu_i\right)$ which is finite since $(\mu_i)_i$ regularly varies with index $\rho>1$.
Hence, $h_2$ regularly varies at $+\infty$ as $i_a$, that is, with index $1/\rho$.

\noindent Putting all the pieces together, $\log \phi=h_1+h_2(1+R)$ is a negative function which regularly varies with index $1/\rho$ at $+\infty$. Then, according to Lemma \ref{resultat_Bingham}, the function
$$t\longmapsto-\log \p_{\infty}(T_0\leq t)$$
regularly varies at $0$ with index $1/(1-\rho)$, which concludes the proof.
\end{proof}

\bigskip

\begin{proof}[Proof of Theorem \ref{descente_rapide}(ii)]
We now suppose that $l=0$ and that $(\mu_n)_n$ regularly varies with index $\rho>1$.
 Let $(E_i)_{i\geq1}$ be a sequence of independent exponential random variables with respective parameters $\lambda_i+\mu_i$.
First, to go from $+\infty$ to $0$, the process $X$ has to reach each integer  level. That gives the upper bound 
\begin{equation}\label{upper_bound}
 \p_{\infty}(T_0\leq t)\leq \p_\infty \left( \sum_{i=1}^{\infty}  E_i \leq t\right),\quad t>0.
\end{equation}
Moreover, for $n\geq1$ and $t\geq0$, we have
\begin{equation*}
\p_n(T_0\leq t)\geq \p_n(X\downarrow)\p_n(T_0\leq t|X\downarrow)
=\prod_{i=1}^{n}\frac{\mu_i}{\lambda_i+\mu_i}\p_n\left(\sum_{i=1}^{n}E_i\leq t\right),
\end{equation*}
where $X\downarrow:=\{\hbox{The process $X$  is non-increasing}\}$. Indeed, conditionally on $X\downarrow$, $(X(t) : t\geq 0)$ is a pure death process with death rates $\mu_n=\lambda_n+\mu_n$.

\noindent For $\eta \in (0,1)$ and $t>0$, we denote $n_{t}:= v(t)(1+\eta)$.
\noindent Hence, for $t>0$, by applying the Markov property, we get the lower bound
{\setlength\arraycolsep{2pt}
\begin{eqnarray}
 \p_\infty(T_0\leq t)&\geq&\p_\infty(X(t/2)\leq n_{t/2})\p_{n_{t/2}}(T_0\leq t/2)\nonumber\\
 &\geq& \p_{\infty}(X(t/2)\leq n_{t/2})\prod_{i=1}^{n_{t/2}}\frac{\mu_i}{\lambda_i+\mu_i} \p_\infty \left(\sum_{i=1}^{n_{t/2}} E_i \leq t/2\right).  \label{lower_bound}
\end{eqnarray}}
Putting together \eqref{upper_bound} and \eqref{lower_bound}, at a logarithmic scale, we obtain
{\setlength\arraycolsep{2pt}
\begin{eqnarray}
-\log\p_\infty \left( \sum_{i=1}^{\infty}  E_i \leq t\right)&\leq&-\log\p_\infty(T_0\leq t)\leq-\log\p_\infty \left( \sum_{i=1}^{n_{t/2}}  E_i \leq t/2\right)\label{inegalites_proba2}\\
 &&\qquad\qquad\qquad\qquad\qquad+\sum_{i=1}^{n_{t/2}}\log\left(1+\frac{\lambda_i}{\mu_i}\right)-\log\p_{\infty}( X_{t/2} \leq n_{t/2}).\nonumber
\end{eqnarray}}
We know from the pure death case that
$$t\mapsto\log\p_\infty \left( \sum_{i=1}^{n_{t/2}}  E_i \leq t/2\right) \textrm{ and }
t\mapsto\log\p_\infty \left( \sum_{i=1}^{\infty}  E_i \leq t\right)$$
both regularly vary at $0$ with common index $1/(1-\rho)$. Moreover,
according to Theorem \ref{Theoprincipal}(i),
$$\log\p_\infty( X(t/2) \leq n_{t/2})=\log \p_\infty\left(\frac{X(t/2)}{v(t/2)}\leq 1+\eta\right)\underset{t\to0}\longrightarrow0.$$
Let us deal with the remaining term of \eqref{inegalites_proba2}, namely $\sum_{i=1}^{n_{t/2}}\log(1+\lambda_i/\mu_i)$. Since $l=0$, $(\lambda_n/\mu_n)_n$ is bounded
and $\sum_{i=1}^{n_{t/2}}\log(1+\lambda_i/\mu_i)\leq C n_{t/2}$ for some $C>0$.
Then, $\sum_{i=1}^{n_{t/2}}\log(1+\lambda_i/\mu_i)$ is upper bounded by a regularly varying function with index $1/(1-\rho)$ since $n(t/2)=v(t/2)(1+\eta)$.

\me
\noindent Plugging our four estimates into  \eqref{inegalites_proba2} ensures that
there exist two slowly varying functions $\underline{l}$ and $\overline{l}$ such that
$$t^{1/(1-\rho)}\underline{l}(t)\leq-\log\p_\infty(T_0\leq t)\leq t^{1/(1-\rho)}\overline{l}(t).$$
So 
for $t<1$
$$\frac{\log \underline{l}(t)}{\log t}\geq\frac{\log(-\log\p_\infty(T_0\leq t))}{\log t}-\frac{1}{1-\rho}\geq \frac{\log \overline{l}(t)}{\log t}.$$
Adding that the right and left hand sides tend to zero according to Lemma \ref{croissances_comparees}(ii), the proof is complete.
\end{proof}

\section{Application to inhomogeneous birth and death processes}\label{application}

Thanks to the previous results, we can estimate the probability of extinction for birth and death processes. An estimation of  the probability of extinction before a small time  
comes from Theorem \ref{descente_rapide}  and an estimation to be extincted after a  large time $t$ can be obtained from the
exponential moments obtained in  Proposition \ref{CDI} by Markov inequality. \\

\noindent As an application, we can now state some asymptotic results for population dynamics  in varying environment.
We consider a time inhomogeneous birth and death process $\,(Z_{t}, t\geq 0)$, associated with a varying environment. The birth and death rates at time $t$ are respectively
$\lambda_k(t)$ and $\mu_k(t)$ when the population size is equal to $k$.
We say that an environment is favorable if the process persists with positive probability in this environment. \\
In case of random environment and under uniform assumptions on the birth and death rates, which make that either all the environments are favorable or all the environments are non favorable, extinction criteria are known, see e.g. Theorem 3.2 in \cite{Torrez1978} and Theorem 3.3 in \cite{Cogburn1981}.\\
Here we consider a case where favorable and non favorable environments can be mixed, successively in time. The assumption below focuses on time intervals of unfavorable environments.

\paragraph{Assumption A.} There exists a sequence of successive and disjoint time intervals $\,([a_i,a_i+t_i), i\in \mathbb{N})\,$ such that  for each $t\in \cup_{i\geq 0}\, [a_i,a_i+t_i)$,
$$\lambda_k(t) \leq  \lambda_k,  \qquad \mu_k(t)\geq \mu_k,$$
where $\lambda_k/\mu_k\to0$ as $k\to+\infty$ and $(\mu_k)_k$ regularly varies with index $\rho>1$.\\

\noindent  Assumption A means that on the successive intervals $[a_i,a_i+t_i)$, the environment is unfavorable and strongly increases the sub-criticality of the process. The intervals  $[a_i,a_i+t_i)$ can be seen as  competition phases. We will show that  under Assumption A and  even if the population is very large at time $a_{i}$, the process can go down to extinction during the time   $[a_i,a_i+t_i)$, even for some  $t_{i}$ tending  to zero  and whatever happens during the phases $]a_i+t_i, a_{i+1}]$. Nevertheless,   the durations $t_{i}$ cannot go to zero  too fast and we provide a quantitative criterion.   Such a framework is  motivated from ecology  by the fact that favorable environments  can alternate with unfavorable environments. It may be due to variations of the climate conditions and the resources available, which can affect
both the natality, the mortality and  the competition. One particular motivating example is a case with linear birth rates  and $\rho>1$ corresponding to a polynomial competition term  ($\rho=2$ yields the logistic competition). See \cite{Sibly22072005} for discussions about the value of $\rho$.\\
\noindent  Let us also remark that the environment may also model the effect of some predation, when the dynamics of the predator  does not depend on the number of preys (generalist predator),  see e.g. \cite{Coffey}.
As a last motivation, we mention the use of inhomogeneous birth and death processes in  epidemiology, see for example \cite{BroekHees,BacaerDads} in the linear case.
The two forthcoming results give the minimal duration of the competition phases which leads to a.s. extinction.

\medskip
\begin{prop} \label{compmin} Under Assumption A, 
if  there exists $\beta <\rho-1$ such that, for $i$ large enough and some constant $c>0$,
\be \label{ti} t_i\geq c/\log^\beta i,\ee
then for every $n\in\mathbb N$,  the process $Z$ becomes extinct  in finite time $\mathbb P_n$-a.s.
\end{prop}

\begin{proof} We  denote by $T_0$ the absorbing time of a process with birth rate $\lambda_k$
and death rate $\mu_k$ and $\widetilde \p_{\infty}$ its law issued from infinity (which is well defined since $l=0$ and $(\mu_k)_k$ regularly varies with index $\rho>1$).
By the Markov property and a monotonicity argument, a simple induction yields  for every $n\geq 0$, $k\geq 1$,
$$\p_n(Z_{a_k+t_k}>0)= \mathbb E_n\left(\p_n \left(Z_{a_k+t_k}>0 \left|    \ (Z_s,  s\leq a_k)  \right.\right) \1{Z_{a_{k-1}+t_{k-1}}>0} \right)\leq \prod_{i \leq k } \widetilde \p_{\infty}(T_0> t_i),$$
so that
$$\log  \p_n(\forall t>0 : Z_t>0)\leq \sum_{i\geq 0} \log \widetilde \p_{\infty}(T_0> t_i)= \sum_{i\geq 0} \log \mathbb (1-\widetilde \p_{\infty}(T_0 \leq t_i)).$$
We know from Theorem  \ref{descente_rapide}(ii) that for every $\eta>0$ and for $t$ small enough,
$$\widetilde  \p_{\infty}(T_0\leq t)\geq \exp\left(-t^{-\left(\frac{1}{\rho-1}+ \eta\right)}\right).$$
Since
$\widetilde \p_{\infty}(T_0\leq t_{i})\geq \widetilde \p_{\infty}(T_0\leq c/ \log^\beta i)$, we get that for any $\eta>0$  and $i$ large enough,
$$\widetilde \p_{\infty}(T_0\leq t_{i})\geq \exp\left(-{\left(\frac{\log^\beta i}{c}\right)^{\frac{1}{\rho -1} +\eta}}\right).$$
By hypothesis, there is $\eta$ small enough such that 
$\beta (\frac{1}{\rho-1}+ \eta)<1$. 
Therefore, $\left(\frac{\log^\beta i}{c}\right)^{\frac{1}{\rho -1} +\eta} < \log i\, $ for $\,i\,$ large enough and
$$ \sum_{i\geq 0} \exp\left(-{\left(\frac{\log^\beta i}{c}\right)^{\frac{1}{\rho -1} +\eta}}\right)=\infty.$$
Finally,  $\p_n(\forall t>0 : Z_t>0)=0$, which  ends up the proof.
\end{proof}

\noindent  The result of Proposition \ref{compmin} is completed by the following example where  \eqref{ti} fails and where the process survives with positive probability. Let us define the birth and death process $Z$  as follows. We assume   that $a_{0}=0$ and that  $\lambda_k(t) = \lambda>0$ for any $k$ (for simplicity). Moreover,   
  for every $t\in \cup_{i\geq 0}\, [a_i,a_i+t_i)$,
$\mu_k(t)= \mu_k$   regularly varies with index $\rho>1$ and
 for every $t\in \cup_{i\geq 0} [a_i+t_i,a_{i+1})$, $\mu_k(t)=0.$

\noindent 
We assume that $t_{i}\leq c/\log^\beta i$ with $\beta>\rho-1$.
Let $\eta>0$ be such that $\beta (\frac{1}{\rho-1} -\eta)>1$.
We now prove that the sequence  $(a_i)_{i\geq 1}$ can be chosen such that the  process  $Z$ survives with positive probability.

\noindent   Let $\epsilon_i \in (0,1)$
such that
$\prod_{i\geq 0} (1-\epsilon_i)>0$.
Let $x_0=1$ and for each $i\geq 1$, choose $x_i\in \mathbb N$ such that 
$$\p_{x_i}(T_0 > t_i)\geq 1-(1-\epsilon_i)\p_{\infty}(T_0>t_i),$$
so that for $i$ large enough, by Theorem \ref{descente_rapide}(ii),
$$\p_{x_i}(T_0>t_i)\geq 1- \epsilon_i\exp\left(-t_i^{-(\frac{1}{\rho-1} -\eta)}\right).$$
Then,
$$\prod_{i\geq 0}\p_{x_i}(T_0>t_i) \geq  \prod_{i\geq 0}\left(1- (1-\epsilon_i)\exp\left(-t_i^{-(\frac{1}{\rho-1} -\eta)}\right)\right)>0.$$
Observing that the process $Z$ is a pure birth process during the time intervals $[a_i+t_i,a_{i+1})$, one can choose the times $a_i$ $(i\geq1)$ such that for every $i\geq 0$,
$\p_1(Z_{a_{i+1}-a_i-t_i} \geq x_{i+1})\geq (1-\epsilon_i)$.
By Markov property for every $n\geq 1$, it yields
$$\p_n(\forall t \geq 0 : Z_t >0)\geq    \prod_{i\geq 0}  \p_{x_i}(T_0>t_i) \mathbb P_1(Z_{a_{i+1}-a_i-t_i} \geq x_{i+1})>0$$
and ends up the proof.

\appendix
\section{Regular varying functions}\label{fonctions_variation_reguliere}
In this section, we give several results that deal with regularly varying functions. The interested reader can see \cite{Bingham1989} for more details.
\begin{defi}
\begin{enumerate}[{\normalfont(i)}]
 \item A function $g:[0,+\infty)\to(0,+\infty)$ has \emph{regular variations} at $L\in\{0,+\infty\}$ if there exists $\rho$ such that for all $a>0$,
 $$\lim_{x\to L}\frac{g(ax)}{g(x)}=a^\rho.$$
\item  A sequence of real non-zero numbers $(u_n)_{n\geq0}$ regularly varies if there exists $\rho$ such that for all $a>0$
  $$\lim_{n\to +\infty}\frac{u_{[an]}}{u_n}=a^\rho$$
  where $[\cdot]$ is the floor function.
\end{enumerate}
In both cases,  the real number $\rho$ is called the \emph{index} and  
when $\rho=0$, one says that variations are \emph{slow}. 
 \end{defi}
\noindent
According to \cite[Thm. 1.9.5]{Bingham1989}, $(u_n)_n$ regularly varies if and only if the function $x\mapsto u_{[cx]}$ regularly varies at $+\infty$. Then, all the following results that we state for regularly varying functions also hold for regularly varying sequences. \\

 \noindent  Regularly varying functions can be compared with power functions as it is recorded in the following proposition. 
 \begin{lem}\label{croissances_comparees}
  Let $g$ be a slowly varying function at $L\in\{0,+\infty\}$, with index $\rho$.
 \begin{enumerate}[{\normalfont(i)}]
 \item
 For all $\eps>0$, $\lim_{x\to L}g(x)/(x^{\rho-\eps})=L$ and 
  $\lim_{x\to L}g(x)/(x^{\rho+\eps})=1/L$, 
  where we use the convention $1/(+\infty)=0$ and $1/0=+\infty$.
\item $\lim_{x\to L}\log g(x)/\log x=\rho.$
\end{enumerate}
  \end{lem}

 \begin{proof}
  We prove the first point when $L=+\infty$. By the definition of a regularly varying function, for $a>0$ and $x$ large enough,
  $g(ax)\geq g(x)a^\rho/2$. Then, for $\eps>0$,
  $$\frac{g(ax)}{(ax)^{\rho-\eps}}\geq a^\eps \frac{g(x)}{2x^{\rho-\eps}}$$
  and by letting $a\to+\infty$, we obtain $\liminf_{n\to+\infty}g(x)/(x^{\rho-\eps})=+\infty$.
  
 \noindent   The second point stems from the first point. Indeed, if for instance $L=+\infty$, for all $\eps>0$ and $t$ large enough, we have $t^{\rho-\eps}\leq g(t)\leq t^{\rho+\eps}$ and at a logarithmic scale, we have the result. The proof is the same if $L=0$.
 \end{proof}

\noindent  In the two following results, one sees that the class of regularly varying functions is stable by inversion and summation.
 \begin{lem}{\cite[Thm 1.5.12]{Bingham1989}}\label{VR_inverse}
  Let $g$ be a regularly varying function at $+\infty$ with index $\alpha>0$. Then, the generalized inverse $$g^{-1}(t):=\inf\{s\geq0;\ g(s)\geq t\},\quad t\geq0$$ is well-defined, regularly varies at $+\infty$ with index $1/\alpha$ and as $t\to+\infty$, $$g(g^{-1}(t))\sim g^{-1}(g(t))\sim t.$$
  The same result holds if $g$ regularly varies at 
  $0$ with a negative index and if $g^{-1}(t)=\inf\{s\geq0;\ g(s)\leq t\}$.
 \end{lem}
\begin{lem}\label{VR_restes_sommes}
  Let $g$ be a function that regularly varies at $+\infty$ with index $\rho<-1$. Then $R(n) = \sum_{k\geq n}g(k)$ regularly varies with index $\rho+1$ and 
 $$\sum_{k\geq n}g(k) \equinf{n} -\frac{ng(n)}{\rho+1}.$$
 \end{lem}
 \begin{proof}
 We only prove the first point since the proof of the second one uses similar arguments.
   First, since $\rho<-1$, according to Lemma \ref{croissances_comparees}, $\sum_{k\geq0}g(k)$ and $\int_0^{+\infty}g(x)\dif x$ are both convergent. Moreover, 
   thanks to \cite[Thm 1.5.3]{Bingham1989}, any regularly varying function with a negative index is equivalent to a non-increasing function. Then, without loss of generality, one can suppose that $g$ is non-increasing. If $I_n:=\int_n^{+\infty}g(x)\dif x$, since $g$ is now non-increasing, a classical comparison between series and integrals entails that
 $1-\frac{g(n)}{I_n}\leq \frac{ R_n}{I_n}\leq1.$
Using that $g$ regularly varies, according to \cite[Thm 1.5.11]{Bingham1989}, 
  \begin{equation}\label{eq:1536}
   \lim_{n\to+\infty}\frac{ng(n)}{I_n}=-(\rho+1).
   \end{equation}
Hence, from the last two displays, we get $I_n\sim R_n$ as $n\to+\infty$.
 We also see from \eqref{eq:1536} that $I$ regularly varies at $+\infty$ with index $1+\rho$. Since $I$ and $R$ are equivalent, $R$ also regularly varies with the same index.
 \end{proof}

 \noindent We end this section by giving two results that involve regularly varying functions. The second one is a Tauberian theorem, which is a key result in the proof of Theorem \ref{descente_rapide}(i).

 \begin{lem}\label{lemme_equivalents}
Let $x_0\in[0,+\infty]$ and let $f$ and $g$ be two positive functions such that
$$f(x) \underset{x\to x_0}\longrightarrow L\in\{0,+\infty\}, \qquad \frac{f(x)}{g(x)} \underset{x\to x_0}\longrightarrow 1.$$
If $h$ regularly varies  at $L$, then
$$\frac{h(f(x))}{h(g(x))}\underset{x\to x_0}\longrightarrow 1.$$
Moreover, if $f(x)=f(x,t)=g(x)(1+t\eps(x))$ with $\lim_{x\to x_0}\eps(x)=0$, the previous convergence holds uniformly in $t$ in any compact subset of $\R$.
\end{lem}

 \begin{proof}
We only prove the case $L=0$. Let fix $\eps>0$.
By definition of a regularly varying function, there exist $\eta,\eta' >0$ such that for every $a \in [1-\eta, 1+\eta]$ and $y \in (0,\eta')$,
$$1-\eps \leq \frac{h(ay)}{h(y)} \leq 1+\eps.$$
Furthermore, for $x$ close enough to $x_0$, we have $g(x)\leq \eta'$ and $(1-\eta)\leq f(x)/g(x) \leq (1+\eta)$,
so that $$\left|\frac{h\left(g(x)\cdot \frac{f(x)}{g(x)}\right)}{h(g(x))}-1\right|\leq\eps,$$ which ends up the first part of the proof.
The second part follows in the same way since $1+t\eps(x)$ goes to $1$ uniformly in $t$ in any compact set.
\end{proof}

 \begin{lem}{\cite[Thm. 4.12.9]{Bingham1989}}\label{resultat_Bingham}
 Let $\nu$ be a positive measure on $(0,+\infty)$ whose Laplace transform
 $$\phi(a):=\int_0^\infty e^{-ax}\dif \nu(x)$$
 converges for all $a>0$. Let $\rho<1$.
 Then, $a\mapsto-\log\phi(a)$ regularly varies at $+\infty$ with index $\rho$
 if and only if 
 $x\mapsto-\log\nu(0,x]$ regularly varies at $0$ with index $\rho/(\rho-1)$.
\end{lem}

\section{Proof of Technical results} 

We first prove the equivalence of \eqref{condition_CDI_mu} and $S<\infty$ under the assumptions \eqref{condition_limite} and \eqref{condition_croissance}.
\begin{lem}\label{equivalence_CDI}
 The series
\begin{equation}\label{deux_sommes}
S=\sum_{i\geq1}\pi_i+\sum_{n\geq1}\frac{1}{\lambda_n\pi_n}\sum_{i\geq n+1}\pi_i \qquad \text{and} \qquad \displaystyle\sum_{n\geq1}\frac{1}{\mu_n}
 \end{equation}
have the same behavior.
\end{lem}
\begin{proof}  First, according to \eqref{condition_limite}, as $n\to+\infty$, $\lambda_n/\mu_{n}\to l<1$ and the first term of the r.h.s. of \eqref{deux_sommes} converges.
  It remains to study the convergence of the series $\sum_{n\geq1}A_n$ where for $n\geq1$,
  $$A_n:=\sum_{j\geq1}\frac{\lambda_{n+1}\cdots\lambda_{n+j-1}}{\mu_{n+1}\cdots\mu_{n+j}}.$$
We have $A_n\geq1/\mu_{n+1}$ since it is the first term of the sum.
Moreover, according to assumption \eqref{condition_limite}, for $n$ large enough, we have
$\lambda_n / \mu_n\leq l':=(1+l)/2$ and
$$A_n\leq\sum_{j\geq1}l'^{j-1}\frac{1}{\mu_{n+j}}\leq\frac{1}{1-l'}\frac{M}{\mu_{n+1}}$$
where $M=\sup \{\mu_{n}/ \mu_{j+n} :j,n\geq1 \}$ is finite thanks to assumption \eqref{condition_croissance}. Putting all pieces together, for $n$ large enough, we have 
$$\frac{1}{\mu_{n+1}}\leq A_n\leq \frac{1}{1-l'}\frac{M}{\mu_{n+1}}$$ and the series $\sum_n\frac{1}{\mu_{n}}$ and $\sum_n A_n$ have the same behaviors.
\end{proof}

\noindent Now we consider   a bounded sequence  $(u_n)_{n\geq1}$ satisfying for every $n\geq 1$
 \begin{equation}\label{recurrence}
u_n=a_nu_{n+1}+b_n, \qquad \text{where} \quad a_n\geq 0, \quad \limsup_{n\to+\infty}a_n<1, \quad b_n>0.
 \end{equation}
We prove the following result used in particular in our work to control the moments of $\tau_n$.

\medskip
\begin{lem}\label{lemme_technique}
Assuming \eqref{recurrence}, then for every  $n\geq1$,
\begin{equation}\label{formule u_n}
u_n=\sum_{i\geq n}b_i\prod_{j=n}^{i-1}a_j \quad \hbox{ (where by convention }  \prod_{j=n}^{n-1}a_j =1).
\end{equation}
Moreover, if $\sup \{ b_{n+k}/b_n : k,n\geq1\}<\infty$,
there exists  $C>0$ such that for every  $n\geq1$
 \begin{equation}\label{encadrement}
 b_n\leq u_n\leq C b_n \quad \hbox{ and thus } \qquad \sup_{k,n\geq1}\frac{u_{n+k}}{u_n}<\infty.
  \end{equation}
If in addition we  assume that $a_n\linf{n}0$, then 
$\ u_n\equinf{n}b_n.$
\end{lem}

\medskip
\begin{proof}
By the recurrence property \eqref{recurrence}, for all $N\geq n\geq1$,
\begin{equation}\label{reccurence2}
 u_n=\sum_{i=n}^Nb_i\prod_{j=n}^{i-1}a_j+u_{N+1}\prod_{j=n}^Na_j.
 \end{equation}
Since $\limsup a_n<1$ and $(u_n)_n$ is bounded, the second term of the r.h.s. of \eqref{reccurence2} vanishes as $N\to+\infty$ and \eqref{formule u_n} is proved.

\noindent We now suppose $K:=\sup \{ b_{n+k}/b_n : k,n\geq1\}<+\infty$. Moreover $\limsup_n a_n<1$ ensures that there exists  $\eps>0$ and $n_0$ such that for $n\geq n_0$, $a_n\leq 1-\eps$.
So writing \eqref{formule u_n} as
\begin{equation}\label{eq:12}
u_n=b_n\left(1+\sum_{i\geq1}\frac{b_{i+n}}{b_n}\prod_{j=n}^{i+n-1}a_{j}\right),
\end{equation}
 Then, for $n$ large enough,
$$b_n\leq u_n\leq b_n\left(1+K\sum_{i\geq1}(1-\eps)^{i-1}\right),$$
which ends the proof of \eqref{encadrement}.
If we also have $\lim_{n\to+\infty}a_n=0$, we obtain $u_n\sim b_n$ as $n\to+\infty$ from \eqref{eq:12} thanks to the dominated convergence theorem.
\end{proof}

\section{Examples}\label{Exemples}
In this paragraph, we give examples that fulfill the Assumptions \eqref{condition_limite}, \eqref{condition_croissance}  and \eqref{condition_CDI_mu}. They   illustrate the law of large numbers 
and central limit theorems of this paper. Special attention is payed to the examples motivated by population dynamics, such as Example $1$.
For these  motivations and convenience, we assume here  that the birth  rate satisfies $\lambda_n\leq C. n$ for some constant $C>0$ and every $n\in \mathbb N$. It captures the linear branching rate and 
allows for example to take into account 
cooperation for small populations, as Allee effect. We are also focusing on the case  $l=0$, which means that the death rate prevails for large population  owing to 
the  competition. Let us note from   Lemma \ref{lemme_moments_tau_n}
that   such assumptions ensure that   $\Einf{\tau_n}\sim 1/\mu_{n+1}$ as $n\to+\infty$. 
\\

\begin{example} We assume that $\mu_n=n^\rho\log^\gamma n$ with either $\rho>1$ or $\rho=1$ and $\gamma>1$. This death rate regularly varies with index $\rho$, so
 that  the  almost sure convergence   of Theorem \ref{Behavior_Tn} holds.  Thus, $T_n$ satisfies a strong law of large numbers with speed
$$\Einf{T_n}\equinf{n}\sum_{k\geq n+1}\frac{1}{k^\rho\log^\gamma k}\equinf{n}\left\{
\begin{array}{ll}
\frac{1}{(\rho-1)n^{\rho-1}\log^\gamma n}&\textrm{ if }\rho>1\\
\frac{1}{(\gamma-1)\log^{\gamma-1} n}&\textrm{ if }\rho=1
\end{array}\right..$$
Moreover, since $\Vinf{T_n}\sim 1/[(2\rho-1)n^{2\rho-1}\log^{2\gamma} n]$ as $n\to+\infty$,
according to Theorem \ref{TLC_Tn}, $T_n$ satisfies the C.L.T.
\begin{gather*}
\frac{\sqrt{2\rho-1}}{\rho-1}\sqrt{n}\left\{(\rho-1)n^{\rho-1}\log^\gamma n\, T_n-1\right\}\overset{\textrm{(d)}}{\linf{n}} N\quad (\rho>1)\\
\frac{\sqrt{n}\log n}{(\gamma-1)}\left\{(\gamma-1)\log^{\gamma-1} n \,T_n-1\right\}\overset{\textrm{(d)}}{\linf{n}} N\quad (\rho=1).
\end{gather*}

\noindent 
Concerning the asymptotic behavior of $X(t)$ as $t\to0$, when $\rho>1$, $v$ regularly varies at $0$ with index $1/(1-\rho)$ and is generally not explicit. However, if $\gamma=0$, that is, if $\mu_n=n^\rho$, we have
$v(t)\sim \left((\rho-1)t\right)^{1/(1-\rho)}$ as $t\to0$.
\end{example}

\begin{example}
Let us illustrate the regimes (ii) of Theorems \ref{Behavior_Tn} and \ref{Theoprincipal}.

\medskip \noindent If $\mu_n=(n!)^\gamma$ with $\gamma>0$, $\Einf{T_n}\sim ((n+1)!)^{-\gamma}$. Hence, $\lim_{n\to+\infty} \Einf{\tau_n}/\Einf{T_n}=1$ and Theorem \ref{Behavior_Tn}(ii) yields
$$((n+1)!)^\gamma T_n\overset{\textrm{(d)}}{\linf{n}}E$$ where $E$ is exponential with parameter $1$.

\medskip \noindent If  $\mu_n=e^{-\beta n}$ with $\beta>0$, $\Einf{T_n}\sim e^{-\beta(n+1)}/(1-e^{-\beta})$. Thus, the conditions of Theorem \ref{Behavior_Tn}(ii) hold true with $\alpha=1-e^{-\beta}$ and
$$e^{\beta(n+1)} T_n\overset{\textrm{(d)}}{\linf{n}}\sum_{k\geq0}e^{-\beta k}E_k$$ where the $E_k$'s are i.i.d. exponential with parameter $1$.
In that case, we can explicit the speed $v$ of Theorem \ref{Theoprincipal} and  we get 
$\,X(t)\sim-(\log t)/\beta\,$ as $\,t\to0$,  in probability.
\end{example}

\begin{example} \label{ctr}
In Theorem \ref{Behavior_Tn}, we did not consider the case where $r_n=\Einf{\tau_n}/\Einf{T_n}$ does not converge. Then one can (only) state  analogous results along the convergent
subsequences.
For instance, if $\mu_{2n}=\mu_{2n+1}=3^{-2n}$, we have 
$$r_{2n}\linf{n}\frac{4}{9} \quad \textrm{ and } \quad r_{2n+1}\linf{n}\frac{4}{5}.$$
Theorem \ref{Behavior_Tn}(ii) still holds in that case, but  for the two subsequences $(T_{2n}/\Einf{T_{2n}})_n$ and $(T_{2n+1}/\Einf{T_{2n+1}})_n$, which converge in distribution to different limits.

\medskip \noindent One can also find examples where $0=\liminf r_n<\limsup r_n$. Then, $(T_n)_n$ has two subsequences satisfying the two regimes (i) and (ii) of Theorem \ref{Behavior_Tn}.
\end{example}

\begin{example}\label{contre_exemple} 
In this last example, we exhibit a sequence of death rates $(\mu_n)_n$ such that the law of large numbers of Theorem \ref{Behavior_Tn}(i) holds in probability but not almost surely.\\
For that purpose, we set $\mu_n=\exp(n/\log n)\log n$. One can check that  $l=0$ and  
$$\Einf{\tau_n}\sim 1/\mu_{n+1}, \qquad \Einf{T_n}\sim S(n+1):=\sum_{k\geq n+1}1/\mu_k$$ as $n\to+\infty$. Moreover, as
 $\mu_n$ is non-decreasing, 
 \begin{equation}
 \int_n^\infty\frac{e^{-x/\log(x)}}{\log x}\dif x\leq  S(n)\leq \int_n^\infty\frac{e^{-x/\log(x)}}{\log x}\dif x+\frac{e^{-n/\log(n)}}{\log n}. \nonumber
  \end{equation}
and
 $$\int_n^\infty\frac{e^{-x/\log(x)}}{\log x}\dif x\equinf{n}\int_n^\infty\left(\frac{1}{\log x}+\frac{1}{(\log x)^2}\right)e^{-x/\log(x)}\dif x=e^{-n/\log(n)}.$$
Combining the two last displays and recalling $r_n=\Einf{\tau_n}/\Einf{T_n}$, we have 
 $$ S(n)\sim \exp(-n/\log n), \qquad
r_n\sim 1/\log n, \qquad r_n\to 0,  \qquad \sum_nr_n^{2}=\infty,$$ 
so that   $T_n/S(n+1)$ goes to $1$ in probability but the almost sure convergence is not guaranteed.

 \noindent Indeed, let us assume now  that  $V_n:=T_n/ S(n+1)$ does converge a.s. toward $1$ and   find a contradiction. We have
 \begin{equation}\label{eq:1658}
 V_{n+1}-V_n=V_{n+1}\left(1-\frac{ S(n+2)}{ S(n+1)}\right)-\frac{\tau_{n}}{ S(n+1)}.
   \end{equation}
By hypothesis, the left hand side of the latter a.s. vanishes as $n\to+\infty$. Moreover, simple computations leads to $S(n+1)/S(n)\to1$ and the first term in 
the r.h.s. of the last display a.s. goes to $0$ since our assumption implies that $V_n$ is bounded a.s.
Hence, putting all pieces together, the term $\tau_{n}/S(n+1)$ of \eqref{eq:1658} has to go to $0$ a.s.\\
 To get a contradiction thanks to Borel-Cantelli's Lemma, it suffices to prove that for $\eps$ small enough, 
$$\sum_{n\geq0}\pinf{\tau_{n} / S(n+1)>\eps}=\infty,$$ 
recalling  that  the r.v. $\tau_n$ are independent.
By a coupling argument, $\tau_{n}$ stochastically dominates $\hat\tau_{n}$ the hitting time of $n$ by a pure death process with death rates $(\mu_{k})_{k}$ and
 starting at $n+1$. In other words, $\tau_n$ is stochastically larger than the    exponential r.v. $\hat \tau_{n}$ with parameter $\mu_{n+1}$.
 Then,
 $\pinf{\tau_{n}/S(n+1)>\eps}\geq \exp(-\eps\mu_{n+1} S(n+1))$. Thanks to previous computations, $\mu_nS(n)\sim\log n$ as $n\to+\infty$. Then, there is $C>0$ such that
 $$\pinf{\tau_{n}/ S(n+1)>\eps}\geq e^{-\eps C\log n}=\frac{1}{n^{C\eps}},$$
 which completes the proof because  $\sum_{n\geq0}\p\left(\tau_{n+1}/ S(n)>\eps\right)$ is infinite as soon as $\eps$ is small enough.
\end{example}

$\newline$

\textbf{Acknowledgement.} 
This work  was partially funded by Chaire Mod\'elisation Math\'ematique et Biodiversit\'e VEOLIA-\'Ecole Polytechnique-MNHN-F.X., the professorial chair Jean Marjoulet, the project MANEGE `Mod\`eles
Al\'eatoires en \'Ecologie, G\'en\'etique et \'Evolution'
09-BLAN-0215 of ANR (French national research agency).

\end{document}